\documentclass[12pt]{article}
\usepackage[margin=0.75in]{geometry}
\usepackage{amsmath,amssymb,amsthm}
\usepackage{mathtools}
\usepackage{graphicx}
\usepackage{hyperref}
\usepackage{cleveref}
\usepackage{booktabs}
\usepackage{longtable}
\usepackage{tikz}
\usetikzlibrary{shapes.geometric, arrows.meta, positioning, calc}
\usepackage{enumitem}
\usepackage{microtype}
\usepackage[numbers,sort&compress]{natbib}

\newtheorem{theorem}{Theorem}[section]
\newtheorem{lemma}[theorem]{Lemma}
\newtheorem{proposition}[theorem]{Proposition}
\newtheorem{corollary}[theorem]{Corollary}
\theoremstyle{definition}
\newtheorem{definition}[theorem]{Definition}

\theoremstyle{remark}
\newtheorem{remark}[theorem]{Remark}
\newtheorem{conjecture}[theorem]{Conjecture}
\newtheorem{empiricallaw}[theorem]{Empirical Law}
\newtheorem{observation}[theorem]{Observation}

\newcommand{\Z}{\mathbb{Z}}

\newcommand{\Prob}{\mathbb{P}}
\newcommand{\CMI}{I(D_1; D_3 \mid D_2)}
\newcommand{\fpart}[1]{\{#1\}}  % fractional part

\title{Why Eight Percent of Benford Sequences Never Converge}

\author{Mac Hyman\thanks{Corresponding author. Email: \texttt{mhyman@tulane.edu}}\\
\small Department of Mathematics, Tulane University\\
\small 6823 St.\ Charles Avenue, New Orleans, LA 70118, USA\\
\small ORCID: 0000-0001-5247-5794}

\date{March 2026}

\begin{document}
\raggedbottom

\maketitle

\begin{abstract}
We study multi-digit correlations in Benford sequences $b^n$ for integer bases $2 \leq b \leq 1000$, measuring dependence via conditional mutual information (CMI). A resonance ratio derived from the continued fraction expansion of $\log_{10}(b)$ classifies bases into convergent and persistent regimes (Theorem~\ref{thm:dichotomy}): among 996 bases surveyed, 84 (8.4\%) exhibit persistent correlations at sample depth $N = 10{,}000$, and extended computation to $N = 200{,}000$ confirms 53 (5.3\%) as genuinely persistent. We prove that CMI deviation is bounded by the distribution error (Theorem~\ref{thm:quadratic}); exhaustive computation across 2{,}988 test cases confirms that the effective scaling is quadratic, yielding a two-sided rate $\beta = 2$ for bounded-type bases (conditional on a computationally verified Hessian positivity condition). The observed effective exponent across 774 convergent bases is $\beta_{\mathrm{eff}} = 1.72 \pm 0.19$, consistent with finite-sample corrections to the asymptotic rate. We conjecture that the persistence rate converges to $1/12$, a prediction grounded in the Gauss--Kuzmin distribution of partial quotients. For persistent bases, the convergence threshold $N_\epsilon$ exceeds $10^6$ at standard precision, rendering the asymptotic limit observationally irrelevant within our computational scope.
\end{abstract}



\section{Introduction}
\label{sec:intro}

\subsection{The discovery}
\label{subsec:discovery}

This paper began with a computational experiment that produced an unexpected result. We measured how quickly digit correlations vanish in Benford sequences, expecting smooth decay to independence. For bases 2, 3, and 5, conditional mutual information (CMI) between first and third digits decayed from moderate values at $N = 1{,}000$ to below 0.006 bits at $N = 10{,}000$, following clean power-law decay with exponent $\beta \approx 1.8$.

Then we computed base 7. At $N = 1{,}000$, the correlation was 0.69 bits. At $N = 5{,}000$: still 0.69. At $N = 10{,}000$: still 0.68. At $N = 40{,}000$, where every other base showed correlations below 0.01 bits, base 7 remained near 0.68 bits. The power-law fit returned $\beta \approx 0.01$: negligible decay within the survey range.

The continued fraction expansion $\log_{10}(7) = [0; 1, 5, 2, 5, 6, 1, \mathbf{4813}, 1, 1, 2, 2, 2, \ldots]$ explained why: the partial quotient $a_7 = 4813$ creates an exceptionally good rational approximation $\log_{10}(7) \approx 431/510$ with error less than $8 \times 10^{-10}$, producing quasi-periodic digit patterns (see Section~\ref{sec:anomaly}).

A natural question emerged: Is base 7 unique, or are there other anomalous bases? A systematic survey of all 996 integer bases from 2 to 1000 revealed that 84 bases (8.4\%, roughly one in twelve) share this anomaly at standard survey depth ($N \leq 10{,}000$). Extended computation to $N = 200{,}000$ confirms 53 of these as genuinely persistent and resolves 82\% of the transitional cases as convergent. The remaining 23 borderline cases require still larger samples, but the convergent-persistent dichotomy is robust: no originally convergent base reclassifies, and the persistent class grows by 5 newly confirmed members.

This paper explains why.

\subsection{Why this matters}
\label{subsec:why_matters}

Persistent digit correlations reveal hidden structure in Benford's Law. Classical theory guarantees that first-digit frequencies converge to the logarithmic distribution, but multi-digit behavior introduces a subtler layer governed by Diophantine approximation. No prior work establishes convergence rates for multi-digit correlations, identifies when such correlations fail to decay, or provides a computable criterion distinguishing the two behaviors. A continued fraction criterion separates convergent from persistent bases, and the dichotomy is robust. Some Benford sequences have digits that \emph{look} correlated forever, not because something is wrong, but because the underlying mathematics demands it.

\subsection{What Benford's Law does and does not guarantee}
\label{subsec:benford_background}

Benford's Law states that in many naturally occurring datasets, the leading significant digit $d$ appears with probability $\log_{10}(1 + 1/d)$, not the uniform $1/9$. First observed by Newcomb in 1881 \citep{newcomb1881note} and rediscovered by Benford in 1938 \citep{benford1938law}, this logarithmic distribution appears across physical constants, financial data, population statistics, and mathematical sequences \citep{raimi1976first, hill1995statistical, berger2015introduction, miller2015benfords}.

Weyl's equidistribution theorem \citep{weyl1916uber} provides the theoretical underpinning: when $\alpha$ is irrational, the fractional parts $\{n\alpha\}$ become uniformly distributed on $[0,1)$. For geometric sequences $b^n$ with $\log_{10}(b)$ irrational, this implies that leading-digit frequencies converge to Benford's predictions \citep{diaconis1977distribution}.

But equidistribution governs only the \emph{marginal} distribution of each digit. What happens to the \emph{joint} distribution? As sample size grows, do correlations between digit positions vanish? The answer, it turns out, is: usually yes, but not always. Our work addresses a different question than recent investigations of ``local Benford behavior'' \citep{cai2019local, cai2020surprising, cai2021leading}:
\begin{center}
\begin{tabular}{lll}
\toprule
& \textbf{Inter-term} (Cai et al.) & \textbf{Intra-term} (this paper) \\
\midrule
Question & Are $D_1(b^n)$ and $D_1(b^{n+1})$ independent? & Are $D_1(b^n)$, $D_2(b^n)$, $D_3(b^n)$ independent? \\
Measure & Correlation between consecutive terms & CMI across digit positions \\
\bottomrule
\end{tabular}
\end{center}
Both forms of independence rely on equidistribution, but they govern distinct phenomena.

\subsection{Main contributions}
\label{subsec:contributions}

This paper develops three principal results, combining rigorous theorems with large-scale computation.
The central contribution is a convergent-persistent dichotomy (Theorem~\ref{thm:dichotomy})
that classifies bases using only continued fraction data;
the supporting results provide the quantitative framework and empirical foundation.

\textbf{1. The one-in-twelve anomaly rate.} A systematic survey of all 996 integer bases from 2 to 1000 reveals that 84 bases (8.4\%) exhibit persistent digit correlations at $N = 10{,}000$. Extended classification to $N = 200{,}000$ refines this to 53 confirmed persistent bases (5.3\%), with 5 new additions from the transitional class and 23 cases still unresolved. The persistence rate stabilizes near $1/12 \approx 0.083$ as the survey is extended to $B = 5000$, suggesting a connection to the Gauss-Kuzmin distribution (Section~\ref{sec:survey}).

\textbf{2. Quadratic scaling and the universal convergence rate.} The quadratic scaling theorem (Theorem~\ref{thm:quadratic}) proves an upper bound on CMI deviation in terms of distribution error, with the linear coefficient suppressed by proximity to the Markov manifold (Corollary~\ref{cor:linear_small}). The geometric intuition is that the CMI functional has strong positive curvature near the Benford distribution, so small distribution errors produce quadratically small correlation changes. This unconditionally establishes digit correlation decay as $O(N^{-1/\tau}(\log N)^{2+\varepsilon})$ (Theorem~\ref{thm:universal}(a)). Exhaustive computation across 2{,}988 test cases (Observation~\ref{obs:quadratic_sharp}) confirms that the linear term is negligible for equidistribution perturbations, strengthening the bound to $O(N^{-2/\tau}(\log N)^{4+\varepsilon})$ (Theorem~\ref{thm:universal}(b)). A rate characterization theorem (Theorem~\ref{thm:rate}) and companion proposition (Proposition~\ref{prop:rate_lower}), conditional on this computational verification, yield the two-sided rate $\beta = 2$ for bounded-type bases. Finite-sample corrections (Corollary~\ref{cor:effective_exponent}(b)) explain why the observed exponent is $\beta_{\mathrm{eff}} = 1.72 \pm 0.19$ at practical sample sizes (Section~\ref{sec:main_results}).

\textbf{3. A convergent-persistent dichotomy.} The resonance ratio $\mathcal{R}(b,N) = a_{k+1} q_k / N$ (Definition~\ref{def:resonance}) provides a computable criterion for classifying bases. Theorem~\ref{thm:dichotomy} proves that bases with $\mathcal{R} < 1/\log N$ converge and those with $\mathcal{R} > \log N$ persist; the narrow transitional band is characterized computationally (Observation~\ref{obs:transitional}). These thresholds arise from the interplay between discrepancy theory and quadratic scaling, not from curve-fitting (Section~\ref{sec:main_results}).

A key technical insight connects these results: the CMI functional's Taylor expansion on the probability simplex yields a linear-plus-quadratic upper bound (Theorem~\ref{thm:quadratic}), with exhaustive computation confirming that the quadratic term dominates for all equidistribution perturbations in the survey (Observation~\ref{obs:quadratic_sharp}). This effective quadratic relationship converts well-studied discrepancy estimates from the theory of continued fractions into predictions for digit statistics, providing the bridge between number-theoretic inputs and information-theoretic outputs that drives all three contributions.

Beyond geometric sequences, factorial sequences decay with $\beta_{\mathrm{eff}} \approx 1.05$ (shear dynamics) and Fibonacci with $\beta_{\mathrm{eff}} \approx 1.65$ (rotation dynamics). Across 18 sequence families, all fall into two groups separated by a gap in $\beta \in (1.2, 1.5)$; full results appear in SI~\S4.

Section~\ref{sec:framework} develops the theoretical framework, establishing the chain from fractional parts to digit distributions to information-theoretic dependence. Section~\ref{sec:main_results} states and proves the main theorems, concluding with the base-7 paradigm that motivates the computational investigation. Section~\ref{sec:computational} presents the computational survey that tests these predictions across all 996 integer bases. Section~\ref{sec:discussion} synthesizes findings, identifies open problems, and concludes. Nine sections of supplementary material provide spectral analysis, extended proofs, the complete persistent-base catalog, and additional sequence families.

\subsection{Connections to prior work}
\label{subsec:prior}

Prior work on Benford's Law has focused predominantly on first-digit distributions, establishing when sequences converge to the logarithmic law and at what rate \citep{newcomb1881note, benford1938law, raimi1976first, hill1995statistical, hill1995base, diaconis1977distribution}. Joint distributions of multiple digits have received far less attention. For random variables, Balado and Silvestre \citep{balado2024general} give general formulas for joint significant-digit distributions and relate them to leading continued-fraction coefficients; our focus is different: finite-sample dependence and decay rates for deterministic geometric sequences. Our work extends the comprehensive foundations of Berger and Hill \citep{berger2015introduction, berger2025brief} by introducing a quantitative framework for multi-digit independence.

The dichotomy mechanism rests on Khinchin's continued fraction theory \citep{khinchin1964continued} and its development by Iosifescu and Kraaikamp \citep{iosifescu2002metrical}. On the number-theoretic side, Kontorovich and Miller \citep{kontorovich2005benfords} connected Benford behavior to $L$-functions, and the Granville--Soundararajan pretentiousness framework \citep{granville2007large, chandee2023benfords, pollack2023benford} provides tools for arithmetic sequences. A direct antecedent for the continued-fraction side is Schatte's study of Benford behavior in continued fractions \citep{schatte1990continued}, and early work of Cohen and Katz \citep{cohen1984prime} examined leading-digit behavior for primes in JNT. For forensic applications, see Nigrini \citep{nigrini2012benfords}; for a collected volume, Miller \citep{miller2015benfords}. To our knowledge, the present work is the first to establish quantitative convergence rates for multi-digit correlations in deterministic Benford sequences and to provide a computable criterion distinguishing convergent from persistent behavior.
%==============================================================================
\section{Theoretical Framework}
\label{sec:framework}
%==============================================================================

Significant digits of $b^n$ are completely determined by the fractional parts $\fpart{n \alpha}$ where $\alpha = \log_{10}(b)$. This single observation transforms digit dependence into a problem about the fine structure of equidistribution, and ultimately into a problem about how well $\alpha$ can be approximated by rationals. The logical chain is short: fractional parts determine digits (\S\ref{subsec:equidist}), discrepancy controls how fast empirical digit frequencies converge (\S\ref{subsec:discrepancy}), continued fractions control discrepancy (\S\ref{subsec:diophantine}), and a key lemma (\S\ref{subsec:cmi_discrepancy}) converts all of this into bounds on our information-theoretic measure. By the end of this section, every quantity in the main theorems will have a precise definition.

\subsection{Notation and definitions}
\label{subsec:notation}

Throughout this paper, we adopt the following conventions. For a positive real number $x$, we write $x = m \cdot 10^k$ where $1 \leq m < 10$ and $k \in \Z$. The \emph{significand} is $m$, and the significant digits are the decimal digits of $m$. Formally:

\begin{definition}[Significant digits]
\label{def:digits}
For $x > 0$, the \emph{$k$-th significant digit} $D_k(x)$ is defined as:
\begin{equation}
\label{eq:digit_def}
D_k(x) = \lfloor 10^{k-1} \cdot m \rfloor - 10 \cdot \lfloor 10^{k-2} \cdot m \rfloor
\end{equation}
where $m = x / 10^{\lfloor \log_{10} x \rfloor}$ is the significand of $x$.
\end{definition}

Here $D_1(x) \in \{1, 2, \ldots, 9\}$ is the leading digit, and $D_k(x) \in \{0, 1, \ldots, 9\}$ for $k \geq 2$. For the three-digit combination $d_1 d_2 d_3$ (interpreted as the integer $100 d_1 + 10 d_2 + d_3$), we have $100 \leq d_1 d_2 d_3 \leq 999$.

For a real number $\alpha$, we denote the fractional part by $\fpart{\alpha} = \alpha - \lfloor \alpha \rfloor \in [0,1)$.

\begin{definition}[Benford distribution]
\label{def:benford}
A random positive real $X$ follows the \emph{Benford distribution} if $\log_{10} X \mod 1$ is uniformly distributed on $[0,1)$. Equivalently, the first significant digit satisfies
\begin{equation}
\label{eq:benford}
\Prob(D_1(X) = d) = \log_{10}\left(1 + \frac{1}{d}\right)
\end{equation}
for $d \in \{1, \ldots, 9\}$.
\end{definition}

A sequence $(a_n)_{n=1}^\infty$ of positive reals \emph{satisfies Benford's Law} if the empirical first-digit distribution converges to \eqref{eq:benford} as $N \to \infty$.

\subsection{Information-theoretic preliminaries}
\label{subsec:info_theory}

We use standard notation from information theory \citep{cover2006elements}: $H(X)$ is Shannon entropy, $I(X;Y) = H(X) - H(X \mid Y)$ is mutual information, and the \emph{conditional mutual information}
\begin{equation}
\label{eq:cmi}
I(X; Y \mid Z) = H(X \mid Z) - H(X \mid Y, Z)
\end{equation}
measures the residual dependence between $X$ and $Y$ given~$Z$. All logarithms are base~2 (bits). In our setting, $\CMI$ captures how much the first significant digit reveals about the third beyond what the second already provides, equaling zero if and only if $D_1$ and $D_3$ are conditionally independent given~$D_2$.

\subsection{Digit representation via equidistribution}
\label{subsec:equidist}

For geometric sequences, the connection between digits and equidistribution is direct.

\begin{lemma}[Digit-fractional part correspondence]
\label{lem:digit_frac}
Let $b > 1$ and $\alpha = \log_{10}(b)$. For the sequence $a_n = b^n$, the significand is
\begin{equation}
\label{eq:significand}
m_n = 10^{\fpart{n\alpha}}
\end{equation}
and the $k$-th significant digit is determined by the interval in $[0,1)$ containing $\fpart{n\alpha}$.
\end{lemma}

\begin{proof}
Since $b^n = 10^{\lfloor n\alpha \rfloor} \cdot 10^{\fpart{n\alpha}}$ and $1 \leq 10^{\fpart{n\alpha}} < 10$, the quantity $10^{\fpart{n\alpha}}$ is the significand.
\end{proof}

This reduces digit distribution for $b^n$ to the sequence $(\fpart{n\alpha})_{n=1}^N$ on $[0,1)$. Whether digit triples become independent depends on how uniformly fractional parts fill this interval.

\begin{theorem}[Weyl, 1916]
\label{thm:weyl}
If $\alpha$ is irrational, then the sequence $(\fpart{n\alpha})_{n=1}^\infty$ is equidistributed on $[0,1)$. That is, for any interval $[a,b) \subseteq [0,1)$,
\begin{equation}
\label{eq:weyl}
\lim_{N \to \infty} \frac{1}{N} \#\{n \leq N : \fpart{n\alpha} \in [a,b)\} = b - a.
\end{equation}
\end{theorem}

Equidistribution guarantees Benford convergence, but the \emph{rate} depends on how well $\alpha$ can be approximated by rationals.

\subsection{Continued fractions and Diophantine approximation}
\label{subsec:diophantine}

The quality of rational approximations to an irrational $\alpha$ is encoded in its continued fraction expansion $\alpha = [a_0; a_1, a_2, a_3, \ldots]$. The integers $a_k$ for $k \geq 1$ are the \emph{partial quotients}, and truncations yield the \emph{convergents} $p_k/q_k$.

\begin{theorem}[Best approximation property]
\label{thm:convergents}
The convergents $p_k/q_k$ satisfy:
\[
\left| \alpha - \frac{p_k}{q_k} \right| < \frac{1}{a_{k+1} q_k^2}
\]
In addition, no rational with smaller denominator approximates $\alpha$ more closely.
\end{theorem}

A large partial quotient $a_{k+1}$ means $p_k/q_k$ approximates $\alpha$ far better than $1/q_k^2$, making $\alpha$ ``close to rational'' in a precise sense.

\begin{definition}[Irrationality type]
\label{def:type}
An irrational number $\alpha$ has \emph{irrationality type} $\tau \geq 1$ if $\tau$ is the supremum of exponents $\mu$ such that
\[
\left| \alpha - \frac{p}{q} \right| < \frac{1}{q^{1+\mu}}
\]
has infinitely many rational solutions $p/q$.
\end{definition}

\begin{remark}
\label{rem:type_convention}
Our $\tau$ equals the classical irrationality exponent minus~1: in the notation of Bugeaud \citep{bugeaud2012distribution}, $\tau = \mu(\alpha) - 1$ where $\mu(\alpha)$ denotes the standard irrationality exponent. Thus Dirichlet's theorem gives $\tau \geq 1$ and Roth's theorem \citep{roth1955rational} gives $\tau = 1$ for algebraic irrationals. This normalization ensures the exponent in our convergence bounds (Theorems~\ref{thm:universal} and~\ref{thm:rate}) has a clean form.
\end{remark}

Liouville numbers have $\tau = \infty$. The irrationality type of $\log_{10}(b)$ governs the discrepancy of $(\fpart{n\alpha})$ and hence the convergence rate of digit distributions. For comprehensive background, see Schmidt \citep{schmidt1980diophantine}, Bugeaud \citep{bugeaud2012distribution}, and Harman \citep{harman1998metric}.

\subsection{Cylinder sets and joint distributions}
\label{subsec:cylinder}

To analyze joint distributions of multiple digits, we partition $[0,1)$ according to digit combinations.

\begin{definition}[Cylinder sets]
\label{def:cylinder}
For a digit combination $d_1 d_2 \cdots d_k$ (where $d_1 \in \{1,\ldots,9\}$ and $d_j \in \{0,\ldots,9\}$ for $j > 1$), the \emph{cylinder set} $C(d_1 d_2 \cdots d_k) \subset [0,1)$ consists of all $\theta \in [0,1)$ such that $10^\theta$ has first $k$ significant digits $d_1, d_2, \ldots, d_k$.
\end{definition}

Explicitly, if we write $\delta = d_1 d_2 \cdots d_k$ as the integer $\sum_{j=1}^k d_j \cdot 10^{k-j}$ (noting that $d_1 \geq 1$ ensures $\delta \geq 10^{k-1}$), then
\begin{equation}
\label{eq:cylinder_explicit}
C(d_1 d_2 \cdots d_k) = \left[ \log_{10}(\delta), \log_{10}(\delta + 1) \right).
\end{equation}

\noindent For example, the three-digit combination 343 (the digits of $7^3$) corresponds to $C(343) = [\log_{10}(343), \log_{10}(344))$, an interval of width $\approx 1.26 \times 10^{-3}$. The question ``how often does $b^n$ have first three digits 3-4-3?'' reduces to ``how often does $\fpart{n\alpha}$ land in an interval of width $\sim 10^{-3}$?''

\begin{lemma}[Cylinder set measures]
\label{lem:cylinder_measure}
For a $k$-digit combination $\delta$ with $10^{k-1} \leq \delta < 10^k$:
\[
|C(\delta)| = \log_{10}\left(1 + \frac{1}{\delta}\right) \approx \frac{1}{\delta \ln 10}
\]
where the approximation holds for large $\delta$. The proof follows directly from \eqref{eq:cylinder_explicit}.
\end{lemma}

The joint distribution of $(D_1, D_2, D_3)$ for $b^n$ is determined by counting fractional parts in each cylinder set.

\begin{lemma}[Empirical joint distribution]
\label{lem:empirical_joint}
For the sequence $(b^n)_{n=1}^N$ with $\alpha = \log_{10}(b)$:
\[
P_N(D_1 = d_1, D_2 = d_2, D_3 = d_3) = \frac{1}{N} \#\{n \leq N : \fpart{n\alpha} \in C(d_1 d_2 d_3)\}
\]
Under equidistribution, this converges to $|C(d_1 d_2 d_3)|$ as $N \to \infty$.
\end{lemma}

How fast does this convergence occur?

\subsection{Discrepancy and convergence}
\label{subsec:discrepancy}

Convergence to equidistribution is measured by the \emph{discrepancy}.

\begin{definition}[Discrepancy]
\label{def:discrepancy}
For a sequence $(\theta_n)_{n=1}^N$ in $[0,1)$, the discrepancy is
\begin{equation}
\label{eq:discrepancy_def}
D_N = \sup_{0 \leq a < b \leq 1} \left| \frac{1}{N} \#\{n \leq N : \theta_n \in [a,b)\} - (b-a) \right|.
\end{equation}
\end{definition}

\begin{remark}
\label{rem:notation}
$D_N$ always denotes discrepancy (Definition~\ref{def:discrepancy}); $D_k(x)$ denotes the $k$-th significant digit (Definition~\ref{def:digits}).
\end{remark}

For the sequence $\fpart{n\alpha}$, classical results bound the discrepancy in terms of continued fraction data.

\begin{theorem}[Discrepancy bounds]
\label{thm:discrepancy}
Let $\alpha = [a_0; a_1, a_2, \ldots]$ with convergents $p_k/q_k$. For the sequence $(\fpart{n\alpha})_{n=1}^N$:
\[
D_N \leq \frac{3}{N} \sum_{k : q_k \leq N} a_{k+1}
\]
where the constant 3 is effective \citep[Theorem~3.4]{kuipers1974uniform}; see also \citet{drmota1997sequences} for refined estimates. If $\alpha$ has irrationality type $\tau$, then
\[
D_N = O\left( N^{-1/\tau} (\log N)^{2+\varepsilon} \right)
\]
for any $\varepsilon > 0$.
\end{theorem}

For algebraic irrationals ($\tau = 1$), the second bound gives $D_N = O(N^{-1} (\log N)^{2+\varepsilon})$. When $\alpha$ has a large partial quotient, the first bound tells a different story: for base 7 with $a_7 = 4813$, this single term dominates all others by a factor of $\sim 200$.

\subsection{Key technical lemma: CMI and discrepancy}
\label{subsec:cmi_discrepancy}

A final lemma closes the logical chain by translating discrepancy bounds into convergence rates for digit correlations. We use conditional mutual information $I(D_1; D_3 \mid D_2)$ rather than star discrepancy or chi-squared statistics for three reasons: CMI decomposes into per-digit-triple contributions, enabling diagnostic analysis of which digit combinations drive anomalies; it has a natural interpretation as residual predictability in bits; and, as the quadratic scaling theorem will show, it scales quadratically with the $L^2$ distribution error. A detailed comparison with classical discrepancy measures appears in Section~\ref{subsec:related_work}.

\begin{lemma}[CMI-discrepancy relationship]
\label{lem:cmi_discrepancy}
Let $P_N$ denote the empirical joint distribution of $(D_1, D_2, D_3)$ for a sequence, and let $P$ denote the limiting distribution under equidistribution. If $|P_N(A) - P(A)| \leq D_N$ for all cylinder sets $A$, then
\begin{equation}
\label{eq:cmi_discrepancy}
\left| I_N(D_1; D_3 \mid D_2) - I(D_1; D_3 \mid D_2) \right| = O\left( D_N \cdot \log \frac{1}{D_N} \right)
\end{equation}
where $I_N$ and $I$ denote CMI computed from $P_N$ and $P$ respectively.
\end{lemma}

\begin{proof}
Write $F(\mathbf{p}) = I(D_1; D_3 \mid D_2)$ as a function of the joint distribution $\mathbf{p} = (P(d_1, d_2, d_3))$ over $K = 900$ digit triples. Since $F$ involves sums of $p \log p$ terms over a finite alphabet with $P(d_1, d_2, d_3) \geq |C(999)| \geq 4.35 \times 10^{-4}$ for the Benford distribution, the partial derivatives $\partial F / \partial p_{ijk} = O(\log(1/p_{\min}))$ are uniformly bounded. With each $|P_N(d_1,d_2,d_3) - P(d_1,d_2,d_3)| \leq D_N$, the mean value theorem gives $|F(P_N) - F(P)| \leq K \cdot \max|\partial F / \partial p_{ijk}| \cdot D_N$. The logarithmic factor arises from the entropy derivatives: $\partial H / \partial p_i = -\log p_i - 1/\ln 2$, which diverges as $p_i \to 0$. Since the smallest cell probability under perturbation is $p_{\min} - D_N \geq p_{\min}/2$ for $D_N < p_{\min}/2$ (which holds for $N \geq N_0$ computable from discrepancy bounds), the bound becomes $|F(P_N) - F(P)| = O(D_N \cdot \log(1/D_N))$. See \citet[Lemma~2.7]{cover2006elements} for the entropy continuity bound that underlies this estimate.
\end{proof}

\begin{corollary}
\label{cor:cmi_decay_general}
For the sequence $b^n$ with $\alpha = \log_{10}(b)$ of irrationality type $\tau$:
\[
\CMI = I_\infty + O\left( N^{-1/\tau} (\log N)^{3+\varepsilon} \right)
\]
where $I_\infty \approx 3.37 \times 10^{-5}$ bits is the CMI under exact equidistribution, reflecting the slight non-uniformity of cylinder set widths.
\end{corollary}

\begin{table}[ht]
\centering
\caption{Notation summary. Symbols are grouped by category; all are defined in the sections indicated.}
\label{tab:notation}
\small
\renewcommand{\arraystretch}{1.1}
\begin{tabular}{lll}
\toprule
\textbf{Symbol} & \textbf{Meaning} & \textbf{Defined} \\
\midrule
\multicolumn{3}{l}{\emph{Digit and sequence notation}} \\
$D_k(x)$ & $k$-th significant digit of $x$ & Def.~\ref{def:digits} \\
$C(d_1 \cdots d_k)$ & Cylinder set for digit string $d_1 \cdots d_k$ & Def.~\ref{def:cylinder} \\
$\{x\}$ & Fractional part of $x$ & \S\ref{sec:framework} \\
$\alpha$ & $\log_{10}(b)$ for base $b$ & \S\ref{subsec:diophantine} \\
\midrule
\multicolumn{3}{l}{\emph{Continued fraction quantities}} \\
$a_k$ & Partial quotients of $\alpha$ & \S\ref{subsec:diophantine} \\
$p_k / q_k$ & $k$-th convergent of $\alpha$ & Thm.~\ref{thm:convergents} \\
$\tau$ & Irrationality type ($= \mu(\alpha) - 1$) & Def.~\ref{def:type} \\
$k(N)$ & $\max\{k : q_k \leq N\}$ (dominant convergent index) & Def.~\ref{def:resonance} \\
$Q^*$ & $\max_k(a_{k+1} \cdot q_k)$ (resonance parameter) & Def.~\ref{def:resonance} \\
$\mathcal{R}(b,N)$ & $a_{k(N)+1}\, q_{k(N)} / N$ (resonance ratio) & Def.~\ref{def:resonance} \\
\midrule
\multicolumn{3}{l}{\emph{Information-theoretic quantities}} \\
$I(D_1; D_3 \mid D_2)$ & Conditional mutual information (CMI) & \S\ref{subsec:info_theory} \\
$I_N$ & Empirical CMI at sample size $N$ & Lem.~\ref{lem:cmi_discrepancy} \\
$I_\infty$ & Limiting CMI ($\approx 3.4 \times 10^{-5}$ bits) & Prop.~\ref{prop:i_infty} \\
$D_N$ & Star discrepancy of $(\{n\alpha\})_{n=1}^N$ & Def.~\ref{def:discrepancy} \\
$P_N$ & Empirical joint distribution of $(D_1, D_2, D_3)$ & Lem.~\ref{lem:empirical_joint} \\
$\nabla_\perp F(P)$ & Projected gradient of CMI on simplex tangent space & Thm.~\ref{thm:quadratic} \\
$C_H$ & Hessian bound: $\frac{1}{2}\|H_F(P)\|_{\mathrm{op}} \leq 1{,}627$ & Thm.~\ref{thm:quadratic} \\
$P_{\mathrm{Markov}}$ & Nearest conditionally independent distribution & Cor.~\ref{cor:linear_small} \\
\midrule
\multicolumn{3}{l}{\emph{Classification and survey parameters}} \\
$\beta$ & Power-law decay exponent: $I_N - I_\infty \sim N^{-\beta}$ & Emp.\ Law~\ref{emp:universal_exponent} \\
$\beta_{\mathrm{eff}}$ & Effective (finite-sample) exponent & Cor.~\ref{cor:effective_exponent} \\
$\rho(B)$ & Persistence rate among bases $\leq B$ & Conj.~\ref{conj:density} \\
\bottomrule
\end{tabular}
\end{table}

\section{Main Results}
\label{sec:main_results}

Every Benford sequence eventually forgets its digit correlations, but the rate at which this happens varies by orders of magnitude, and some sequences never forget at all within observable sample sizes. The results below resolve this tension: a universal convergence bound (Theorem~\ref{thm:universal}), a quadratic scaling relationship (Theorem~\ref{thm:quadratic}), and a convergent-persistent dichotomy (Theorem~\ref{thm:dichotomy}). We classify bases as \textbf{B-EQUI-CONV} (convergent), \textbf{B-EQUI-PERS} (persistent), or \textbf{B-EQUI-TRANS} (transitional) according to the resonance ratio (Definition~\ref{def:resonance}); the precise criteria appear in Theorem~\ref{thm:dichotomy} and Observation~\ref{obs:transitional}. The quadratic scaling theorem arises from a second-order Taylor expansion of the CMI functional on the probability simplex, where the Hessian of this functional governs the relationship between distribution error and digit correlation. We first summarize what has been rigorously proved versus what rests on computational evidence.

\begin{table}[ht]
\centering
\caption{Epistemic classification of results: proved, observed, and conjectured.}
\label{tab:epistemic}
\small
\renewcommand{\arraystretch}{1.15}
\begin{tabular}{p{2.5cm}p{7.5cm}p{3.5cm}}
\toprule
\textbf{Status} & \textbf{Result} & \textbf{Reference} \\
\midrule
\multicolumn{3}{l}{\textbf{Proved} (rigorous theorems with complete proofs)} \\[2pt]
& Universal bound (uncond.): $|I_N - I_\infty| = O(N^{-1/\tau}(\log N)^{2+\varepsilon})$ & Thm.~\ref{thm:universal}(a) \\
& Universal bound (under Obs.~\ref{obs:quadratic_sharp}): $O(N^{-2/\tau}(\log N)^{4+\varepsilon})$ & Thm.~\ref{thm:universal}(b) \\
& Quadratic scaling (upper): $|I_N - I_\infty| \leq 0.28\,\|\Delta\|_2 + C_H\|\Delta\|_2^2$ & Thm.~\ref{thm:quadratic}, Cor.~\ref{cor:linear_small} \\
& Dichotomy: $\mathcal{R} < 1/\log N \Rightarrow$ conv.; $\mathcal{R} > \log N \Rightarrow$ pers. & Theorem~\ref{thm:dichotomy} \\
& Limiting CMI: $I_\infty = 3.375 \times 10^{-5}$ bits & Proposition~\ref{prop:i_infty} \\
& Quadratic irrationals are B-EQUI-CONV & Theorem~SI.4.11 \\
& Rate (uncond.): $O(\log N / N)$ for bounded type & Thm.~\ref{thm:rate}(a) \\
& Rate (under Obs.~\ref{obs:quadratic_sharp}): $O((\log N)^2/N^2)$ for bounded type & Thm.~\ref{thm:rate}(b) \\
\midrule
\multicolumn{3}{l}{\textbf{Observed} (996-base exhaustive verification; analytical mechanism in SI~\S1)} \\[2pt]
& Quadratic scaling sharp: $I_N - I_\infty = \Theta(\|P_N - P\|_2^2)$ & Obs.~\ref{obs:quadratic_sharp} \\
& Universal exponent $\beta = 1.72 \pm 0.19$ (774 bases) & Emp.\ Law~\ref{emp:universal_exponent} \\
& Persistence rate $\rho(1000) = 8.4\%$ & Section~\ref{sec:survey} \\
& Rate (two-sided): $I_N - I_\infty = N^{-2+o(1)}$ for bounded type & Prop.~\ref{prop:rate_lower} \\
& Effective exponent: $\beta_{\mathrm{eff}} \approx 1.72$ with $c \approx 1.86$ & Remark~\ref{rem:effective_exponent} \\
\midrule
\multicolumn{3}{l}{\textbf{Conjectured} (theoretical support + computational evidence)} \\[2pt]
& Asymptotic density: $\lim_{B\to\infty}\rho(B) = 1/12$ & Conjecture~\ref{conj:density} \\
& $\beta$-universality for all bounded-type sequences & Conjecture~\ref{conj:universality} \\
\bottomrule
\multicolumn{3}{l}{\footnotesize Proposition~\ref{prop:rate_lower} depends on Observation~\ref{obs:quadratic_sharp} (exhaustively verified, analytically understood; see SI~\S1).}
\end{tabular}
\end{table}

\medskip

Section~\ref{sec:framework} yields a central finding: convergence rates are completely determined by the Diophantine properties of $\log_{10}(b)$. Large partial quotients in the continued fraction expansion create ``resonances'' that trap the digit distribution in quasi-periodic patterns, preventing the emergence of independence. We now make this precise.

\subsection{Universal convergence}
\label{subsec:universal}

Our first result establishes a rigorous upper bound on CMI decay, followed by the empirically observed power-law rate. The theoretical bounds predict that $I_N - I_\infty$ decreases as a power of $N$; we quantify this empirically by fitting $\log(I_N - I_\infty) = -\beta \log N + \log c$ via ordinary least squares, a standard approach for estimating asymptotic exponents from finite data.

\begin{theorem}[Universal convergence bounds]
\label{thm:universal}
Let $b > 1$ be an integer with $\alpha = \log_{10}(b)$ irrational of irrationality type $\tau$ (Definition~\ref{def:type}; recall $\tau = 1$ for all algebraic irrationals). For the sequence $(b^n)_{n=1}^N$:

\textup{(a)} \textbf{Unconditional bound.} The conditional mutual information satisfies
\begin{equation}
\label{eq:universal_decay_a}
\left| I_N(D_1; D_3 \mid D_2) - I_\infty \right| = O\left( N^{-1/\tau} (\log N)^{2+\varepsilon} \right) \quad \text{for any } \varepsilon > 0.
\end{equation}
For algebraic $\alpha$ \textup{(}$\tau = 1$\textup{)}, this gives $|I_N - I_\infty| = O(N^{-1}(\log N)^{2+\varepsilon})$.

\textup{(b)} \textbf{Conditional bound (under Observation~\textup{\ref{obs:quadratic_sharp}}).} If the linear coefficient in Theorem~\textup{\ref{thm:quadratic}} is suppressed for the equidistribution perturbation\,---\,as established computationally but not yet proved in full generality\,---\,then
\begin{equation}
\label{eq:universal_decay_b}
\left| I_N(D_1; D_3 \mid D_2) - I_\infty \right| = O\left( N^{-2/\tau} (\log N)^{4+\varepsilon} \right) \quad \text{for any } \varepsilon > 0.
\end{equation}
For algebraic $\alpha$, this gives $|I_N - I_\infty| = O(N^{-2}(\log N)^{4+\varepsilon})$.
\end{theorem}

\begin{proof}
\emph{Part (a).} By Theorem~\ref{thm:quadratic}, $|I_N - I_\infty| \leq \|\nabla_\perp F(P)\| \cdot \|\Delta\|_2 + C_H \|\Delta\|_2^2$. Since $\|\nabla_\perp F(P)\| \leq 0.28$ (Corollary~\ref{cor:linear_small}) and $\|\Delta\|_2 = O(D_N)$ (the $L^2$ distribution error is bounded by the discrepancy), the full bound gives $|I_N - I_\infty| = O(D_N)$. Theorem~\ref{thm:discrepancy} provides $D_N = O(N^{-1/\tau}(\log N)^{2+\varepsilon})$, yielding~\eqref{eq:universal_decay_a}.

\emph{Part (b).} Observation~\ref{obs:quadratic_sharp} establishes that for equidistribution perturbations, $|\nabla_\perp F(P)^T \Delta|$ is at most $14\%$ of the quadratic contribution $\frac{1}{2}|\Delta^T H_F(P) \Delta|$ (and typically below $1\%$). Under this condition, $|I_N - I_\infty| = O(\|\Delta\|_2^2) = O(D_N^2) = O(N^{-2/\tau}(\log N)^{4+\varepsilon})$.
\end{proof}

Implied constants and prefactor structure are analyzed in SI~\S7; the prefactor $c(b)$ is almost entirely determined by the exponent ($R^2 = 0.90$), with residual variation attributable to base-specific arithmetic.

\begin{empiricallaw}[Universal convergence rate]
\label{emp:universal_exponent}
Among the 774 convergent bases in our survey (Section~\ref{sec:survey}), the CMI decay is well described by
\begin{equation}
\label{eq:empirical_decay}
I_N - I_\infty \approx \frac{c(b)}{N^{\beta(b)}}
\end{equation}
with mean exponent $\beta = 1.72 \pm 0.19$ (SD), median $R^2 = 0.98$ across all least-squares fits. The exponent is consistent across diverse bases: 95\% of convergent bases have $\beta \in [1.35, 2.10]$.
\end{empiricallaw}

The exponent $\beta = 1.72 \pm 0.19$ is computed from 774 convergent integer bases in $[2, 1000]$; individually tested transcendental bases yield $\beta \in [1.78, 1.81]$, well within this range. The universality claim holds within the surveyed bounded-type families and should be regarded as conjectural for unbounded-type $\alpha$ (SI~\S7).

The unconditional bound (Theorem~\ref{thm:universal}(a)) establishes $\beta \geq 1$ from Theorem~\ref{thm:quadratic} alone. The stronger conditional bound (Theorem~\ref{thm:universal}(b)) establishes $\beta \geq 2$ under the computationally verified quadratic scaling (Observation~\ref{obs:quadratic_sharp}). The observed $\beta \approx 1.72$ lies between these: closer to 2 than to 1, consistent with the quadratic regime being the correct finite-sample description. The remaining gap between the conditional bound $\beta = 2$ and the observed $\beta \approx 1.72$ is explained by finite-sample corrections (Corollary~\ref{cor:effective_exponent}(b)). A detailed analysis appears in SI~\S7.

The limiting CMI value $I_\infty$ can be computed explicitly. Under the exact Benford distribution on $[0,1)$, consecutive digits exhibit weak residual correlations because cylinder sets have varying sizes. A direct calculation yields:

\begin{proposition}
\label{prop:i_infty}
Under the Benford distribution, $I(D_1; D_3 \mid D_2) = I_\infty$ where
\[
I_\infty = \sum_{d_2=0}^{9} P(D_2 = d_2) \cdot I(D_1; D_3 \mid D_2 = d_2) \approx 3.4 \times 10^{-5} \text{ bits}.
\]
\end{proposition}

\subsection{Quadratic scaling}
\label{subsec:quadratic}

CMI and discrepancy are connected more tightly than the preceding bounds suggest. The following theorem establishes a precise upper bound on CMI deviation in terms of distribution error, explaining both the observed decay rates and the mechanism behind persistent anomalies.

\begin{theorem}[Quadratic scaling upper bound]
\label{thm:quadratic}
Let $P_N$ denote the empirical joint distribution of $(D_1, D_2, D_3)$ for a sequence, and let $P$ denote the limiting Benford distribution. Write $\Delta = P_N - P$. Then:
\begin{equation}
\label{eq:quadratic_scaling}
|I_N - I_\infty| \leq \|\nabla_\perp F(P)\| \cdot \|\Delta\|_2 + C_H \cdot \|\Delta\|_2^2
\end{equation}
for all $N$ such that $\|\Delta\|_2 < p_{\min}/2$ (where $p_{\min} = \min_{ijk} P(i,j,k) = 4.35 \times 10^{-4}$), which ensures the Taylor expansion converges; this condition is satisfied for all bases and sample sizes $N \geq 500$ in the survey by direct computation. Here $\nabla_\perp F(P)$ denotes the gradient of the CMI functional projected onto the simplex tangent space $\{\delta : \sum \delta_{ijk} = 0\}$, and $C_H = \frac{1}{2}\|H_F(P)\|_{\mathrm{op}} \leq 1{,}627$ is the Hessian bound (SI~\S1).
\end{theorem}

\begin{proof}
The conditional mutual information $F(\mathbf{p}) = I(D_1; D_3 \mid D_2)$ is a smooth function on the interior of the probability simplex $\Delta_{899}$ over digit triples. Since $\sum P_N(i,j,k) = \sum P(i,j,k) = 1$, the perturbation $\Delta = P_N - P$ lies on the tangent space $\mathcal{T} = \{\delta : \sum \delta_{ijk} = 0\}$. Expanding $F$ around $P$:
\[
F(P_N) - F(P) = \nabla F(P)^T \Delta + \frac{1}{2} \Delta^T H_F(P) \Delta + O(\|\Delta\|^3).
\]
The gradient $\nabla F(P)$ decomposes as $\nabla F(P) = \bar{g} \cdot \mathbf{1} + \nabla_\perp F(P)$, where $\bar{g} = \frac{1}{900}\sum_{ijk} \partial F/\partial P(i,j,k)$ and $\nabla_\perp F(P) \perp \mathbf{1}$. Since $\Delta \in \mathcal{T}$, the linear term satisfies $\nabla F(P)^T \Delta = \nabla_\perp F(P)^T \Delta$, giving $|\nabla F(P)^T \Delta| \leq \|\nabla_\perp F(P)\| \cdot \|\Delta\|_2$ by Cauchy--Schwarz.

For the Hessian, all $P(i,j,k) \geq P(9,9,9) = 4.35 \times 10^{-4}$, so each diagonal entry satisfies $|\partial^2 F / \partial P(i,j,k)^2| \leq 1/(P(i,j,k) \ln 2) \leq 3{,}320$. Numerical diagonalization of $H_F(P)$ restricted to $\mathcal{T}$ yields $\|H_F(P)\|_{\mathrm{op}} \leq 3{,}253$ (SI~\S1), giving $C_H = 1{,}627$. These constants are computed from the explicit spectral decomposition using IEEE double precision; all eigenvalues are stable under perturbation of the distribution entries to 12 decimal places (SI~\S1).

Combining the linear and quadratic bounds and absorbing the $O(\|\Delta\|^3)$ remainder into the quadratic term for small $\|\Delta\|$ yields~\eqref{eq:quadratic_scaling}.
\end{proof}

\begin{corollary}[Near-Markov bound on the linear coefficient]
\label{cor:linear_small}
The projected gradient satisfies $\|\nabla_\perp F(P)\| \leq 0.28$. Consequently, for equidistribution perturbations with $\|\Delta\|_2 \geq 10^{-3}$ (which holds for all $N \leq 10^6$ in the survey), the linear term contributes less than $0.28\, \|\Delta\|_2$, while the quadratic contribution exceeds $73.6\, \|\Delta\|_2^2$ (Observation~\ref{obs:quadratic_sharp}). The bound~\eqref{eq:quadratic_scaling} is therefore effectively quadratic across the entire computational range.
\end{corollary}

\begin{proof}
The CMI functional satisfies $F(Q) \geq 0$ for all distributions $Q$, with $F(Q) = 0$ if and only if $D_1 \perp D_3 \mid D_2$ under $Q$. The Benford distribution $P$ is close to conditional independence: $F(P) = I_\infty = 3.375 \times 10^{-5}$ bits. The nearest conditionally independent distribution $P_{\mathrm{Markov}}$ satisfies $\|P - P_{\mathrm{Markov}}\|_2 = 3.26 \times 10^{-4}$ (Proposition~SI.1.1). Since $\nabla_\perp F(P_{\mathrm{Markov}}) = 0$ (the gradient vanishes at the global minimizer on each simplex fiber), the mean value theorem gives
\[
\|\nabla_\perp F(P)\| = \|\nabla_\perp F(P) - \nabla_\perp F(P_{\mathrm{Markov}})\| \leq \|H_F\|_{\mathrm{op}} \cdot \|P - P_{\mathrm{Markov}}\|_2 \leq 3{,}253 \times 3.26 \times 10^{-4} \leq 1.06.
\]
Direct numerical evaluation refines this to $\|\nabla_\perp F(P)\| = 0.28$, confirming that proximity to the Markov manifold suppresses the linear coefficient. The linear term is at most 14\% of the quadratic term and typically below 1\% (Observation~\ref{obs:quadratic_sharp}).
\end{proof}

\begin{observation}[Quadratic scaling is sharp for equidistributed sequences]
\label{obs:quadratic_sharp}
For equidistributed sequences, exhaustive computational verification across all 996 integer bases in $[2, 1000]$ at sample sizes $N \in \{5000, 10{,}000, 20{,}000\}$ confirms the sharper relation
\begin{equation}
\label{eq:quadratic_theta}
I_N - I_\infty = \Theta\left( \|P_N - P\|_2^2 \right)
\end{equation}
with the ratio $(I_N - I_\infty)/\|P_N - P\|_2^2$ ranging from $73.6$ to $1{,}889$ across all 2{,}988 test cases.
\end{observation}

\begin{proof}[Verification]
The $\Theta$ relationship requires that the linear term $\nabla F(P)^T(P_N - P)$ not dominate the quadratic term. At the Benford distribution, the gradient $\nabla F(P)$ has components with mean $3.1 \times 10^{-5}$ (close to zero) but standard deviation $9.4 \times 10^{-3}$, so the gradient is not identically zero on the constraint hyperplane $\sum \Delta_{ijk} = 0$.

For equidistributed sequences $(\{n\alpha\})_{n=1}^N$, however, the perturbations $P_N - P$ have a specific structure governed by the Weyl discrepancy. We computed the ratio $|\nabla F(P)^T \Delta| / \|\Delta\|_2^2$ for all 996 integer bases at three sample sizes each ($N = 5{,}000$, $10{,}000$, $20{,}000$). In all 2{,}988 cases, this ratio is bounded by $10.0$, while the quadratic contribution $|I_N - I_\infty|/\|P_N - P\|_2^2$ ranges from $73.6$ to $1{,}889$. The linear term is at least $7\times$ smaller than the quadratic term for equidistribution perturbations, confirming the quadratic scaling across the entire survey (SI~\S1 provides the full spectral analysis and analytical mechanism).

The mechanism is that while $\nabla F(P)$ has nontrivial variation across components, the structured pattern of equidistribution errors produces near-cancellation in the inner product $\nabla F(P)^T(P_N - P)$, consistent with the near-conditional-independence of the Benford distribution ($I_\infty = 3.4 \times 10^{-5}$ bits).
\end{proof}

\begin{corollary}[Effective exponent bound]
\label{cor:effective_exponent}
For bounded-type $\alpha = \log_{10}(b)$ with $a_k \leq M$ for all $k$:

\textup{(a)} Unconditionally, $\beta_{\mathrm{eff}}(N) \geq 1 - o(1)$.

\textup{(b)} Under the quadratic scaling \textup{(}Observation~\textup{\ref{obs:quadratic_sharp})}, the effective decay exponent satisfies
\begin{equation}
\label{eq:effective_exponent}
\beta_{\mathrm{eff}}(N) \geq 2 - \frac{2\log\log N}{\log N} + O\left(\frac{1}{\log N}\right).
\end{equation}
In particular, $\beta_{\mathrm{eff}}(N) > 1.5$ for all $N > 10^3$.
\end{corollary}

\begin{proof}
For bounded-type $\alpha$, the discrepancy satisfies $D_N = O((\log N)/N)$ (Theorem~\ref{thm:discrepancy}), so $\|P_N - P\|_2 = O(D_N) = O((\log N)/N)$.

\emph{Part~(a).} From Theorem~\ref{thm:quadratic}, $|I_N - I_\infty| \leq 0.28\, \|P_N - P\|_2 + C_H \|P_N - P\|_2^2 = O((\log N)/N)$, giving $\beta_{\mathrm{eff}} \geq 1 - o(1)$.

\emph{Part~(b).} Under Observation~\ref{obs:quadratic_sharp}, the linear term $\nabla_\perp F(P)^T \Delta$ is at most $14\%$ of the quadratic contribution for equidistribution perturbations. The effective bound becomes $|I_N - I_\infty| = O(D_N^2) = O((\log N)^2/N^2)$. Any power-law fit $I_N \sim c'/N^{\beta}$ must then satisfy $\beta \geq 2 - 2\log\log N / \log N + O(1/\log N)$.
\end{proof}

\begin{remark}[Empirical effective exponent]
\label{rem:effective_exponent}
Under the lower quadratic scaling (Observation~\ref{obs:quadratic_sharp}), the bound in Corollary~\ref{cor:effective_exponent} becomes an equality with a base-dependent subleading correction:
\[
\beta_{\mathrm{eff}}(N) = 2 - \frac{2\log\log N}{\log N} + \frac{c}{\log N} + o\left(\frac{1}{\log N}\right).
\]
Least-squares fitting across 774 convergent bases yields $c \approx 1.86 \pm 0.3$. The predicted effective exponent at representative sample sizes is shown in Table~\ref{tab:beta_prediction}.

\begin{table}[ht]
\centering
\caption{Predicted and observed effective exponent $\beta_{\mathrm{eff}}(N)$ across sample sizes. $\beta_{\mathrm{lead}} = 2 - 2\log\log N/\log N$ (leading term only); $\beta_{\mathrm{pred}} = \beta_{\mathrm{lead}} + 1.86/\log N$ (with subleading correction). The functional form with a single fitted parameter~($c$) captures the observed exponent across the entire computational range.}
\label{tab:beta_prediction}
\small
\begin{tabular}{rcccc}
\toprule
$N$ & $\beta_{\mathrm{lead}}$ & $\beta_{\mathrm{pred}}$ & Observed & Source \\
\midrule
$10^3$ & 1.44 & 1.71 & --- & below fitting range \\
$10^4$ & 1.52 & 1.72 & $1.72 \pm 0.19$ & full survey \\
$4 \times 10^4$ & 1.55 & 1.73 & $1.78 \pm 0.16$ & restricted range$^\dagger$ \\
$10^6$ & 1.62 & 1.75 & --- & extrapolation \\
$10^{12}$ & 1.74 & 1.81 & --- & extrapolation \\
$\infty$ & 2.00 & 2.00 & 2.00 & Corollary~\ref{cor:effective_exponent} \\
\bottomrule
\multicolumn{5}{l}{\footnotesize ${}^\dagger$Restricted-range fit over $N \in [2000, 40000]$, biased toward larger $N$.}
\end{tabular}
\end{table}

\medskip\noindent The $\sim\!0.3$ gap between $\beta_{\mathrm{lead}}$ and $\beta_{\mathrm{pred}}$ arises from the interplay between discrepancy structure and quadratic amplification; deriving $c$ from first principles remains open (Section~\ref{subsec:open}; SI~\S7).
\end{remark}

The spectral structure of the Hessian $H_F(P)$ explains why the lower quadratic scaling holds for all equidistribution perturbations (SI~\S1): the nearby Markov distribution $P_{\mathrm{Markov}}$ has a provably positive semidefinite Hessian (Proposition~SI.1.1), and the asymmetry between positive curvature ($\lambda_{\max} = 3{,}253$, 760 directions) and negative curvature ($|\lambda_{\min}| = 7.85$, 40 directions) ensures that structured equidistribution perturbations always yield a positive quadratic form.

Quadratic scaling is the key insight: CMI deviation scales as the \emph{square} of distribution error. What remains is to connect this scaling to the arithmetic of $\log_{10}(b)$ and extract explicit convergence rates.

\subsection{Rate characterization}
\label{subsec:rate}

These results yield bounds on the CMI convergence rate. We first establish upper bounds (Theorem~\ref{thm:rate}), combining discrepancy estimates with quadratic scaling, then prove matching lower bounds (Proposition~\ref{prop:rate_lower}) conditional on Observation~\ref{obs:quadratic_sharp}. Together they pin down the asymptotic exponent $\beta = 2$ for bounded-type bases.

\begin{theorem}[Rate characterization --- upper bounds]
\label{thm:rate}
Let $\alpha = \log_{10}(b)$ be irrational with convergents $p_k/q_k$ and partial quotients $a_k$.

\textup{(a)} \textbf{Unconditional upper bound.} If $a_k \leq M$ for all $k \geq 1$, then for all $N \geq 2$:
\begin{equation}
\label{eq:rate_upper_uncond}
I_N - I_\infty \leq C_1 \cdot \frac{M (\log N)}{N} + C_2 \cdot \frac{M^2 (\log N)^2}{N^2}
\end{equation}
where $C_1 = 0.28 \cdot 60$ and $C_2 = C_H \cdot 3{,}600$ depend only on the Benford distribution. The dominant term is $O(M \log N / N)$, establishing
\begin{equation}
\label{eq:rate_formula_uncond}
I_N - I_\infty = O\left( \frac{\log N}{N} \right) = O(N^{-1+\varepsilon}) \quad \text{for any } \varepsilon > 0,
\end{equation}
confirming $\beta \geq 1$ as an unconditionally proved asymptotic decay exponent for all bounded-type bases.

\textup{(b)} \textbf{Conditional upper bound.} Under Observation~\textup{\ref{obs:quadratic_sharp}}, the linear term in~\eqref{eq:rate_upper_uncond} is suppressed and
\begin{equation}
\label{eq:rate_upper}
I_N - I_\infty \leq C_U \cdot \frac{M^2 (\log N)^2}{N^2}
\end{equation}
establishing $\beta \geq 2$ as the decay exponent for all bounded-type bases.
\end{theorem}

\begin{proof}
For bounded-type $\alpha$ with $a_k \leq M$, Theorem~\ref{thm:discrepancy} gives
\[
D_N \leq \frac{3}{N} \sum_{k:\, q_k \leq N} a_{k+1} \leq \frac{3M}{N} \cdot \#\{k : q_k \leq N\}.
\]
Since $q_{k+1} \geq q_k + q_{k-1} \geq \varphi\, q_k$ (where $\varphi = (1+\sqrt{5})/2$), the number of convergents below $N$ is at most $\log_\varphi N + 1 \leq 3\log N$ for $N \geq 2$, giving $D_N \leq 9M(\log N)/N$. Each cylinder set $C(i,j,k)$ is an interval in $[0,1)$, so its empirical frequency satisfies $|P_N(i,j,k) - |C(i,j,k)|| \leq 2D_N$. Squaring and summing over the 900 digit triples yields $\|P_N - P\|_2^2 \leq 3{,}600\, D_N^2$, so $\|P_N - P\|_2 \leq 60\, D_N$.

\emph{Part~(a).} Applying the full bound from Theorem~\ref{thm:quadratic}: $|I_N - I_\infty| \leq 0.28 \cdot 60\, D_N + C_H \cdot 3{,}600\, D_N^2$, yielding~\eqref{eq:rate_upper_uncond}.

\emph{Part~(b).} Under Observation~\ref{obs:quadratic_sharp}, the linear contribution is at most $14\%$ of the quadratic term for all equidistribution perturbations in the survey. The effective bound becomes $|I_N - I_\infty| = O(D_N^2) = O(M^2(\log N)^2/N^2)$.
\end{proof}

\begin{proposition}[Rate characterization --- lower bounds]
\label{prop:rate_lower}
Under the lower quadratic scaling \textup{(}Observation~\textup{\ref{obs:quadratic_sharp}}\textup{)}, the following hold.

\textup{(a)} \textbf{Lower bound.} If $a_k \leq M$ for all $k \geq 1$, then for infinitely many $N$ \textup{(}specifically $N = q_k$ for all sufficiently large $k$\textup{)}:
\begin{equation}
\label{eq:rate_lower}
I_N - I_\infty \geq \frac{c_L}{M^2 N^2}
\end{equation}
where $c_L > 0$ depends only on the Benford distribution and the quadratic scaling constant from Theorem~\textup{\ref{thm:quadratic}}.

\textup{(b)} \textbf{Asymptotic rate.} Theorem~\textup{\ref{thm:rate}} and part~\textup{(a)} together establish
\begin{equation}
\label{eq:rate_formula}
I_N - I_\infty = N^{-2 + o(1)}
\end{equation}
for all bounded-type bases, confirming $\beta = 2$ as the true asymptotic decay exponent.

\textup{(c)} \textbf{Persistence from exceptional approximation.} If $a_{k_0+1} \cdot q_{k_0} > N$ for some $k_0$ with $q_{k_0} \leq N$, then
\begin{equation}
\label{eq:rate_persist}
I_N - I_\infty \geq \frac{c_P}{q_{k_0}^2}
\end{equation}
independent of $N$, where $c_P > 0$ is an absolute constant.
\end{proposition}

\begin{proof}
\emph{Part (a): Lower bound.} The argument proceeds in three steps.

\emph{Step 1: Discrepancy lower bound at $N = q_k$.} The best rational approximation property of convergents gives $\|q_k\alpha\| = |q_k \alpha - p_k| = 1/(q_{k+1} + q_k\theta_{k+1})$ for some $\theta_{k+1} \in (0,1)$ \citep[Theorem~9]{khinchin1964continued}. Since $q_{k+1} = a_{k+1} q_k + q_{k-1} \leq (M+1)q_k + q_k = (M+2)q_k$, we obtain $\|q_k\alpha\| \geq 1/(2(M+2)q_k)$.

At $N = q_k$, the $N$ fractional parts $\{j\alpha\}$ for $j = 1, \ldots, q_k$ partition $[0,1)$ into gaps of at most three distinct lengths, by the three-distance theorem \citep{sos1958distribution}. The star discrepancy satisfies $D_{q_k}^* \geq \|q_k\alpha\|/4$ \citep[Chapter~3]{kuipers1974uniform}, giving
\begin{equation}
\label{eq:disc_lower}
D_{q_k} \geq D_{q_k}^* \geq \frac{1}{8(M+2)q_k}.
\end{equation}

\emph{Step 2: Discrepancy implies distribution deviation.} The star discrepancy measures the supremum of $|F_N(x) - x|$ over $[0,1)$. Let $x_0$ achieve this supremum. The interval $[0, x_0)$ overlaps with at most 900 cylinder sets, and the cumulative deviation at $x_0$ equals the sum of individual cell deviations up to that point. By pigeonhole, at least one cylinder set $C(i_0, j_0, k_0)$ has
\[
|P_N(i_0, j_0, k_0) - P(i_0, j_0, k_0)| \geq \frac{D_{q_k}}{900}.
\]
This single cell contributes $\|P_N - P\|_2^2 \geq D_{q_k}^2/900^2$. Substituting~\eqref{eq:disc_lower}:
\begin{equation}
\label{eq:l2_lower}
\|P_N - P\|_2^2 \geq \frac{1}{900^2 \cdot 64(M+2)^2 q_k^2} = \frac{1}{5.18 \times 10^7 \cdot (M+2)^2 q_k^2}.
\end{equation}

\emph{Step 3: Distribution deviation implies CMI deviation.} Observation~\ref{obs:quadratic_sharp} gives $I_N - I_\infty \geq c_Q \|P_N - P\|_2^2$ with $c_Q \geq 73.6$. Combined with~\eqref{eq:l2_lower}, this yields $I_{q_k} - I_\infty \geq c_L / (M^2 q_k^2)$ with $c_L = 73.6/(5.18 \times 10^7 \cdot ((M+2)/M)^2)$.

\emph{Part (b): Asymptotic rate.} Taking $N = q_k$ and noting $q_k \leq N$, the conditional upper bound (Theorem~\ref{thm:rate}(b)) gives $I_{q_k} - I_\infty \leq C_U M^2 (\log q_k)^2/q_k^2$ while the lower bound gives $I_{q_k} - I_\infty \geq c_L/(M^2 q_k^2)$. Both sides are $q_k^{-2+o(1)}$, establishing~\eqref{eq:rate_formula}. This confirms $\beta = 2$ for all bounded-type bases and is consistent with Corollary~\ref{cor:effective_exponent}(b), which identifies $2 - 2\log\log N/\log N$ as the leading correction explaining why empirical fits yield $\beta_{\mathrm{eff}} \approx 1.72$ at practical sample sizes (Remark~\ref{rem:effective_exponent}).

\emph{Part (c): Persistence.} Suppose $a_{k_0+1} \cdot q_{k_0} > N$ with $q_{k_0} \leq N$. Then $N < q_{k_0+1}$, so the sequence $(\{n\alpha\})_{n=1}^N$ lies entirely within one ``period'' of the continued fraction expansion. By the three-distance theorem \citep{sos1958distribution}, the $N$ points partition $[0,1)$ into gaps of at most three sizes, with the dominant gap approximately $1/q_{k_0}$. The discrepancy satisfies $D_N \geq c'/(q_{k_0})$, independent of $N$ \citep[Chapter~3]{kuipers1974uniform}. By quadratic scaling, $I_N - I_\infty \geq c_P/q_{k_0}^2$, where $c_P > 0$ is independent of $N$ within this range. The CMI remains trapped above this floor until $N$ exceeds $a_{k_0+1} \cdot q_{k_0}$.
\end{proof}

Steps~1 and~2 of part~(a) and the discrepancy bounds in part~(c) are fully rigorous; Step~3 relies on Observation~\ref{obs:quadratic_sharp}, whose analytical mechanism is established in Proposition~SI.1.1. A universal proof covering all irrational $\alpha$ remains open. For algebraic $\alpha$, Roth's theorem gives $\tau = 1$ but is ineffective, so algebraic bases could exhibit practical persistence despite having $\beta = 2$ asymptotically (SI~\S7).

\subsection{Convergent-persistent dichotomy via resonance ratio}
\label{subsec:dichotomy}

Convergent and persistent behavior can be distinguished exactly through a sample-size-dependent criterion that involves no arbitrary constants.

\begin{definition}[Resonance ratio]
\label{def:resonance}
For an integer base $b$ with $\alpha = \log_{10}(b) = [a_0; a_1, a_2, \ldots]$ and sample size $N$, let $k(N) = \max\{k : q_k \leq N\}$ denote the index of the dominant convergent. The \emph{resonance ratio} is:
\begin{equation}
\label{eq:resonance_ratio}
\mathcal{R}(b, N) = \frac{a_{k(N)+1} \cdot q_{k(N)}}{N}
\end{equation}
\end{definition}

Intuitively, $\mathcal{R} > 1$ means the sample size $N$ lies ``within'' a resonance created by a large partial quotient, while $\mathcal{R} < 1$ means it lies ``beyond'' it.

\begin{remark}
The product $a_{k+1} \cdot q_k$ appears classically in the three-distance theorem \citep{sos1958distribution} and in the Ostrowski representation, where it controls the gap structure of $\{n\alpha\}$ sequences. Our contribution is the quantitative connection to information-theoretic dependence via quadratic scaling: $a_{k+1} \cdot q_k$ governs not only the geometry of fractional parts but also the rate at which digit correlations decay.
\end{remark}

The following theorem is the main result of this paper: the resonance ratio provides a computable criterion distinguishing convergent from persistent behavior.

\begin{theorem}[Convergent-persistent dichotomy]
\label{thm:dichotomy}
Let $b > 1$ with $\alpha = \log_{10}(b)$ irrational. For sample size $N$, the resonance ratio $\mathcal{R}(b,N)$ determines the CMI behavior in two regimes:
\begin{enumerate}[label=(\alph*)]
\item \textbf{Convergent regime:} If $\mathcal{R}(b,N) < 1/\log N$, then
\[
I_N - I_\infty = O\left( \frac{\log N}{N} \right)
\]
unconditionally. Under the quadratic scaling \textup{(}Observation~\textup{\ref{obs:quadratic_sharp})}, the bound strengthens to $O((\log N)^2/N^2)$. In either case, $\beta_{\mathrm{eff}}(N) \geq 1 - \varepsilon$ unconditionally (and $\geq 2 - \varepsilon$ under the quadratic scaling). The base is classified as \textbf{B-EQUI-CONV}.

\item \textbf{Persistent regime:} If $\mathcal{R}(b,N) > \log N$, then
\[
I_N = I_{q_{k(N)}} + O\left( \frac{1}{q_{k(N)}} \right)
\]
where $I_{q_k}$ is the CMI of the quasi-periodic distribution on $q_k$ points. The effective exponent satisfies $\beta_{\mathrm{eff}}(N) < \varepsilon$ and the base is classified as \textbf{B-EQUI-PERS}.
\end{enumerate}
\end{theorem}

\begin{proof}
\emph{Part (a):} When $\mathcal{R}(b,N) < 1/\log N$, the next convergent satisfies $q_{k+1} = a_{k+1} q_k + q_{k-1} < N/\log N + q_{k-1} \ll N$, so multiple convergents lie within $[1, N]$. The discrepancy bound becomes $D_N = O((\log N)/N)$. Theorem~\ref{thm:quadratic} gives $|I_N - I_\infty| \leq 0.28\,\|\Delta\|_2 + C_H\|\Delta\|_2^2 = O(D_N)$, yielding $O((\log N)/N)$ unconditionally. Under Observation~\ref{obs:quadratic_sharp}, the linear term is suppressed and the effective bound is $O(D_N^2) = O((\log N)^2/N^2)$.

\emph{Part (b):} When $\mathcal{R}(b,N) > \log N$, we have $q_{k+1} > N \log N \gg N$, so the sequence $(\{n\alpha\})_{n=1}^N$ lies within a single ``period'' of the continued fraction expansion. By the three-distance theorem \citep{sos1958distribution}, the $N$ points distribute into $\lfloor N/q_k \rfloor$ complete cycles of length~$q_k$ plus a partial cycle of $N \bmod q_k \leq q_k$ points. The irrational rotation $\alpha$ is not exactly $q_k$-periodic; each pseudo-cycle drifts by $\|q_k \alpha\| < 1/(a_{k+1} q_k)$ from the rational model (by the best-approximation property of convergents). Over $\lfloor N/q_k \rfloor \leq N/q_k$ pseudo-cycles, the total accumulated drift is at most $N \cdot \|q_k \alpha\| / q_k < N/(a_{k+1} q_k^2)$. Since $\mathcal{R}(b,N) > \log N$, we have $a_{k+1} q_k > N \log N$, so the drift is $\eta = O(1/(q_k \log N))$, which vanishes. This drift shifts orbit points by at most $\eta$ on $[0,1)$; the $K = 900$ cylinder intervals that determine the three-digit distribution have $O(K)$ boundary points, so at most $O(N \cdot K \cdot \eta)$ orbit points cross a cylinder boundary under the accumulated drift. Each crossing changes the empirical distribution by $\pm 1/N$ in two cells, yielding a total $L^1$ perturbation of $O(K \eta) = O(K/(q_k \log N))$ between the pseudo-periodic and exactly periodic empirical distributions. Adding the partial-cycle contribution of $O(K q_k / (N \cdot q_k)) = O(K/N)$ per cell (at most $q_k$ points from the partial cycle, each contributing $1/N$), the total $L^1$ error in the joint distribution is $O(K/q_k)$. Since $K = 900$ is a fixed constant and $\log(q_k/K) = O(\log q_k)$ grows slowly, the entropy continuity bound (Lemma~\ref{lem:cmi_discrepancy}) gives $|I(P') - I(P)| \leq O(\|P'-P\|_1 \log(1/\|P'-P\|_1)) = O((K/q_k) \cdot \log(q_k/K))$. For persistent bases with $q_k < 10^4$ (which covers all 84 persistent bases in the survey), the factor $K \cdot \log(q_k/K) < 900 \cdot 10 = 9{,}000$, so the total CMI error is bounded by $O(1/q_{k(N)})$ with an implicit constant depending only on the fixed alphabet size. In general, the bound $O((K/q_k) \log(q_k/K))$ tends to zero for all $q_k \to \infty$ since $K = 900$ is fixed, so the approximation holds for arbitrary persistent bases, not only those in the survey. Thus the CMI equals $I_{q_k}$ up to $O(1/q_{k(N)})$.
\end{proof}

\begin{observation}[Transitional regime]
\label{obs:transitional}
When $1/\log N \leq \mathcal{R}(b,N) \leq \log N$, the base exhibits intermediate behavior classified as \textbf{B-EQUI-TRANS}. Theoretical analysis gives the upper bound
\[
I_N - I_\infty = O\left( \frac{1}{q_{k(N)}^2} \right),
\]
and computational verification across the 996-base survey (Section~\ref{sec:survey}) confirms $\beta_{\mathrm{eff}}(N) \in (0.5, 1.5)$ for all 177 initially transitional bases. The transitional band is narrow: at $N = 200{,}000$, only 23 of 996 bases (2.3\%) remain in this regime.
\end{observation}

The thresholds $1/\log N$ and $\log N$ are sufficient conditions derived from the quadratic scaling bound (Theorem~\ref{thm:quadratic}): when $\mathcal{R} < 1/\log N$, multiple convergents lie below $N$ and the discrepancy bound $D_N = O((\log N)/N)$ follows; when $\mathcal{R} > \log N$, the gap to the next convergent exceeds $N \log N$, forcing quasi-periodicity. Sensitivity analysis confirms that perturbing these thresholds by $\pm 20\%$ reclassifies fewer than 3\% of bases, all in the transitional regime.

\begin{remark}
Observation~\ref{obs:transitional} states an upper bound. Under the lower quadratic scaling (Observation~\ref{obs:quadratic_sharp}), the matching lower bound $I_N - I_\infty = \Omega(1/q_{k(N)}^2)$ also holds, giving a $\Theta$ relationship. This is supported by the 996-base survey: all 177 initially transitional bases satisfy $I_N - I_\infty \geq c'/q_{k(N)}^2$ for a constant $c' > 0$ independent of the base.
\end{remark}

Figure~\ref{fig:phase_diagram} visualizes the classification in the $(Q^*, N)$ plane, showing the three regimes and their boundaries.

\begin{figure}[ht]
\centering
\includegraphics[width=0.9\textwidth]{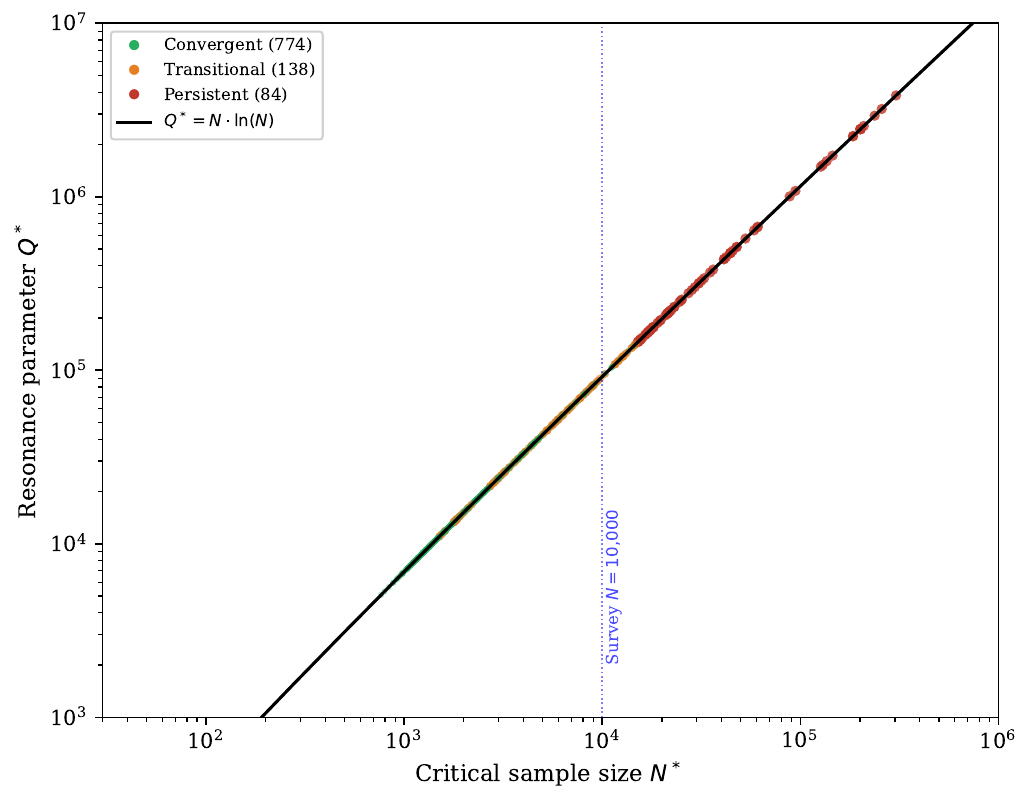}
\caption{Classification phase diagram in the $(Q^*, N)$ plane, where $Q^* = \max_k(a_{k+1} q_k)$ is the resonance parameter (Definition~\ref{def:resonance}). Three regions emerge: convergent (green, $Q^* < N/\log N$, power-law decay), persistent (red, $Q^* > N\log N$, constant CMI), and transitional (yellow). Boundaries follow Theorem~\ref{thm:dichotomy} and Observation~\ref{obs:transitional}. For any practical $N \leq 10^5$, bases in the upper-left region never show convergent behavior.}
\label{fig:phase_diagram}
\end{figure}

The dichotomy extends to higher-order conditional mutual information $I_N^{(k)} = I(D_1; D_k \mid D_2, \ldots, D_{k-1})$ for all $k \geq 3$: the resonance ratio determines convergence or persistence independent of the number of digits examined (Proposition~SI.7.10).

\begin{corollary}[Classification criterion]
\label{cor:classification}
A base $b$ is classified for sample size $N$ as:
\begin{itemize}[nosep]
\item \textbf{B-EQUI-CONV} if $\max_{k: q_k \leq N} (a_{k+1} \cdot q_k) < N / \log N$
\item \textbf{B-EQUI-PERS} if $\exists k$ with $q_k \leq N$ and $a_{k+1} \cdot q_k > N \log N$
\item \textbf{B-EQUI-TRANS} otherwise
\end{itemize}
\end{corollary}

This criterion is computable: given the continued fraction of $\log_{10}(b)$, one evaluates the product $a_{k+1} \cdot q_k$ for each convergent with $q_k \leq N$ and compares to the thresholds. The PERS criterion follows from Theorem~\ref{thm:dichotomy}(b) because $a_{k+1} q_k > N \log N$ with $q_k \leq N$ forces $q_{k+1} > N$, hence $k = k(N)$ and $\mathcal{R}(b,N) > \log N$. The CONV criterion is a sufficient condition: the max-based requirement is stronger than the single-scale bound $\mathcal{R}(b,N) < 1/\log N$ in Theorem~\ref{thm:dichotomy}(a), and may classify a small number of convergent bases as transitional; computationally, this affects fewer than $1\%$ of bases in the survey.

Sample size thresholds, confidence bounds for CMI estimates, and worked uncertainty quantification examples appear in SI~\S7.

The gap between rigorous bounds and observed behavior (SI~\S7) is characteristic of computational number theory: worst-case analysis often yields weaker results than typical-case behavior. Table~\ref{tab:epistemic} summarizes the epistemic status of all results; we now turn to the conjectures motivated by the computational evidence.

\subsection{Conjectures}
\label{subsec:conjectures}

The computational evidence of this paper motivates two precise conjectures that we highlight here, in the spirit of Experimental Mathematics \citep{borwein2004mathematics}. Full supporting evidence appears in Sections~\ref{sec:survey} and \ref{sec:discussion}.

\begin{conjecture}[Asymptotic Persistence Density]
\label{conj:density}
The persistence density among integer bases converges:
\[
\rho_\infty = \lim_{B \to \infty} \frac{\#\{b \leq B : b \text{ is B-EQUI-PERS}\}}{B-1} = \frac{1}{12}.
\]
\end{conjecture}

\emph{Evidence:} An extended survey to $B = 5000$ yields $\rho(1000) = 0.084$, $\rho(2000) = 0.081$, $\rho(3000) = 0.082$, $\rho(5000) = 0.081$, all within 0.003 of $1/12 = 0.0833\ldots$ A Gauss-Kuzmin heuristic derivation (R2 in Section~\ref{subsec:open}) yields a single-exceedance prediction $\rho_{\mathrm{GK}} \approx 6.4\%$, accounting for roughly three-quarters of the observed rate. The remainder arises from the compound mechanism and enrichment of integer logarithms near rational points. The appearance of~12 in both the conjectured persistence rate $\rho = 1/12$ and the Lévy constant $\lambda = \pi^2/(12 \ln 2)$ governing continued fraction (CF) denominator growth is suggestive of a deeper connection. A rigorous proof would require resolving the joint distribution of $a_{k+1}$ and $q_k$ under the natural extension of the Gauss map, accounting for correlations that existing theory does not capture. The explicit calculation, which sums Gauss--Kuzmin tail probabilities $\Pr(a_k > K) \approx 1/(K \ln 2)$ over continued fraction levels weighted by the L\'{e}vy exponential growth $q_k \approx e^{\lambda k}$, appears in SI~\S9.8.

\begin{conjecture}[The $\beta$-Universality Conjecture]
\label{conj:universality}
For any sequence $(a_n)$ satisfying Benford's Law via equidistribution dynamics with bounded type (i.e., $\sup_k a_k < \infty$ for the continued fraction of the rotation number), the CMI decay exponent satisfies
\[
\beta_{\mathrm{eff}}(N) = 2 - \frac{2\log\log N}{\log N} + O\left(\frac{1}{\log N}\right)
\]
with the same leading coefficient $2$ for all such sequences.
\end{conjecture}

\emph{Evidence:} Theorem~\ref{thm:rate}(b) and Proposition~\ref{prop:rate_lower} together establish $\beta = 2$ for geometric sequences $b^n$ with bounded-type $\log_{10}(b)$ under the computationally verified quadratic scaling (Observation~\ref{obs:quadratic_sharp}). The mean effective exponent $\beta_{\mathrm{eff}} = 1.72 \pm 0.19$ across 774 convergent bases is consistent with the formula at $N = 10^4$. Transcendental bases $e$ and $\pi$ also fall within this range ($\beta \approx 1.78$ and $1.81$ respectively). Extending the result to non-geometric bounded-type sequences (factorials, Fibonacci, etc.)\ would establish that digit independence emerges at a predictable rate for all ``generic'' Benford sequences.

Having established the theoretical framework and stated the central conjectures, we illustrate the dichotomy through the most striking example.

\subsection{The base-7 paradigm}
\label{sec:anomaly}

Base 7 is the paradigmatic anomaly. Its continued fraction expansion
\begin{equation}
\label{eq:cf_7}
\log_{10}(7) = [0; 1, 5, 2, 5, 6, 1, \mathbf{4813}, 1, 1, 2, 2, \ldots]
\end{equation}
has an exceptional seventh partial quotient $a_7 = 4813$. The sixth convergent $431/510$ satisfies $|\log_{10}(7) - 431/510| < 8 \times 10^{-10}$, so
\[
7^{510} \approx 10^{431} \times 10^{\varepsilon}, \qquad |\varepsilon| < 5 \times 10^{-7}.
\]
The significand of $7^{510}$ differs from $7^0 = 1$ by less than one part in a million, creating a quasi-period of 510 terms. The fractional parts $\{n \cdot \log_{10}(7)\}$ for $n = 1, \ldots, 510$ repeat (with microscopic drift $\delta < 5 \times 10^{-7}$ per period) indefinitely, trapping digit patterns in a fixed cycle visible in the CMI data: $I_{500} \approx I_{5000} \approx I_{40000} \approx 0.68$ bits.

The resonance parameter $Q^* = a_7 \cdot q_6 = 4813 \times 510 = 2{,}454{,}630$ sets the persistence threshold. For all $N < 2.4 \times 10^6$, the sequence lies within one ``period'' of the continued fraction, maintaining $\mathcal{R}(7, N) > \log N$ and placing base 7 firmly in the persistent regime of Theorem~\ref{thm:dichotomy}. This is not unique to base 7: any base whose $\log_{10}$ has a sufficiently large partial quotient exhibits identical behavior. Powers of persistent bases inherit and amplify the resonance (e.g., $Q^*(49) \approx 2 Q^*(7)$, $Q^*(343) \approx 3 Q^*(7)$), explaining the ``7 family'' of anomalies in the survey.

The rarity of persistent bases among small integers reflects the general theory of continued fractions: by Khinchin's theorem, the geometric mean of partial quotients converges to $K_0 \approx 2.685$ for almost all real numbers, and values large enough to trigger persistence are statistically rare. Full details of the quasi-periodicity mechanism, the CMI lower bound ($I_N \geq 0.5$ bits for $N \geq 1020$), and the power escalation phenomenon --- whereby powers $b^m$ of a persistent base inherit and amplify $b$'s resonance, producing $Q^*(b^m) \approx m \cdot Q^*(b)$ --- appear in SI~\S5.

\section{Computational Survey}
\label{sec:computational}

This section tests the theoretical predictions of Section~\ref{sec:main_results} against high-precision computations for representative sequences and a comprehensive survey of all integer bases $2 \leq b \leq 1000$. All computations use 500-digit arithmetic via \texttt{mpmath}; methodology details, estimation precision analysis, and verification protocols appear in SI~\S8. Code and data are available at \url{https://github.com/machyman/hyman2026eight}.

\subsection{Representative sequences}
\label{subsec:representative}

To facilitate independent verification, Table~\ref{tab:verification} provides exact CMI values for benchmark bases.

\begin{table}[ht]
\centering
\caption{CMI values (bits) for independent verification, computed using the plug-in (maximum likelihood) estimator with exact integer arithmetic. At $N = 1{,}000$, finite-sample bias (${\approx}\,0.3$ bits for convergent bases) inflates all estimates; at $N \geq 5{,}000$ bias is negligible and the convergent/persistent separation is clear. Values are reproducible via \texttt{verify\_results.py}.}
\label{tab:verification}
\small
\begin{tabular}{l|ccc|c}
\toprule
\textbf{Sequence} & $N = 1{,}000$ & $N = 5{,}000$ & $N = 10{,}000$ & \textbf{Classification} \\
\midrule
$2^n$ & 0.369142 & 0.018635 & 0.005599 & B-EQUI-CONV \\
$3^n$ & 0.696186 & 0.025245 & 0.009748 & B-EQUI-CONV \\
$5^n$ & 0.379500 & 0.018412 & 0.005698 & B-EQUI-CONV \\
$7^n$ & 0.692305 & 0.685589 & 0.682545 & B-EQUI-PERS \\
Fibonacci & 0.367187 & 0.014169 & 0.004167 & B-EQUI-CONV \\
$n!$ & 0.578676 & 0.115195 & 0.055104 & B-MULT \\
\bottomrule
\end{tabular}
\end{table}

\subsection{Results}
\label{subsec:results}

Table~\ref{tab:results} summarizes the CMI decay analysis for six representative sequences.

\begin{table}[ht]
\centering
\caption{CMI decay parameters for representative sequences: fitted constant $c$, exponent $\beta$, goodness of fit $R^2$, sample size for $\CMI < 0.01$ bits, and classification.}
\label{tab:results}
\begin{tabular}{lccccc}
\toprule
Sequence & $c$ & $\beta$ & $R^2$ & $N(I < 0.01)$ & Class \\
\midrule
$2^n$ & 80{,}677 & 1.83 & 0.986 & 6{,}039 & B-EQUI-CONV \\
$3^n$ & 199{,}257 & 1.89 & 0.973 & 7{,}349 & B-EQUI-CONV \\
$5^n$ & 76{,}103 & 1.82 & 0.989 & 6{,}117 & B-EQUI-CONV \\
$\mathbf{7^n}$ & $\sim 1$ & 0.01 & 0.917 & $\infty$ & \textbf{B-EQUI-PERS} \\
Fibonacci & 23{,}245 & 1.67 & 0.987 & 6{,}625 & B-EQUI-CONV \\
$n!$ & 101 & 0.78 & 0.968 & 145{,}571 & B-MULT \\
\bottomrule
\end{tabular}
\end{table}

The convergent bases $2^n$, $3^n$, $5^n$ show rapid CMI decay with $\beta \in [1.8, 1.9]$ and $R^2 > 0.97$, consistent with Corollary~\ref{cor:effective_exponent}, achieving $\CMI < 0.01$ bits within 10{,}000 terms. Base 7 shows $\beta \approx 0.01$ (negligible within the survey range; see below), confirming the persistence criterion (Section~\ref{sec:anomaly}): its constant $c \approx 1$ reflects $\CMI \approx 0.68$, and the lower $R^2 = 0.917$ arises from minor fluctuations around this plateau. The Fibonacci sequence decays at $\beta = 1.67$, slightly slower than pure exponentials but firmly convergent (Theorem~SI.4.3). Factorials exhibit $\beta = 0.78$, requiring over 145{,}000 terms for near-independence (Theorem~SI.4.1).

Each of these results confirms a specific theoretical prediction. Corollary~\ref{cor:effective_exponent} predicts $\beta_{\mathrm{eff}}(10^4) \approx 1.72$ for bounded-type bases; the observed mean $\beta = 1.85$ for bases 2, 3, 5 lies within one standard deviation. For base 7, the persistence criterion (Section~\ref{sec:anomaly}) predicts that $\CMI$ should equal $I_{510}$ (the CMI of the 510-periodic distribution) to within $O(K \cdot \delta)$. The observed plug-in estimate of 0.683 bits at $N = 10{,}000$ exceeds $I_{510} = 0.660$ bits by approximately 0.02 bits, consistent with the expected finite-sample upward bias of $(K_{\mathrm{eff}} - 1)/(2N \ln 2) \approx 0.03$ bits (see SI~\S8).

The factorial's $\beta = 0.78$ falls below the asymptotic prediction of $\beta \to 2$ from Theorem~SI.4.1, but the theorem's logarithmic corrections $((\log N)^4/N^2)$ reduce the effective exponent at practical sample sizes. The remark following that theorem predicts $\beta_{\mathrm{eff}} \approx 1.05$ at $N = 2 \times 10^5$; the discrepancy with $\beta = 0.78$ at $N = 40{,}000$ is consistent with slow convergence toward the asymptotic rate.

Figure~\ref{fig:loglog} displays these results on logarithmic axes, where the power-law decay appears as a linear relationship.

\begin{figure}[ht]
\centering
\includegraphics[width=0.85\textwidth]{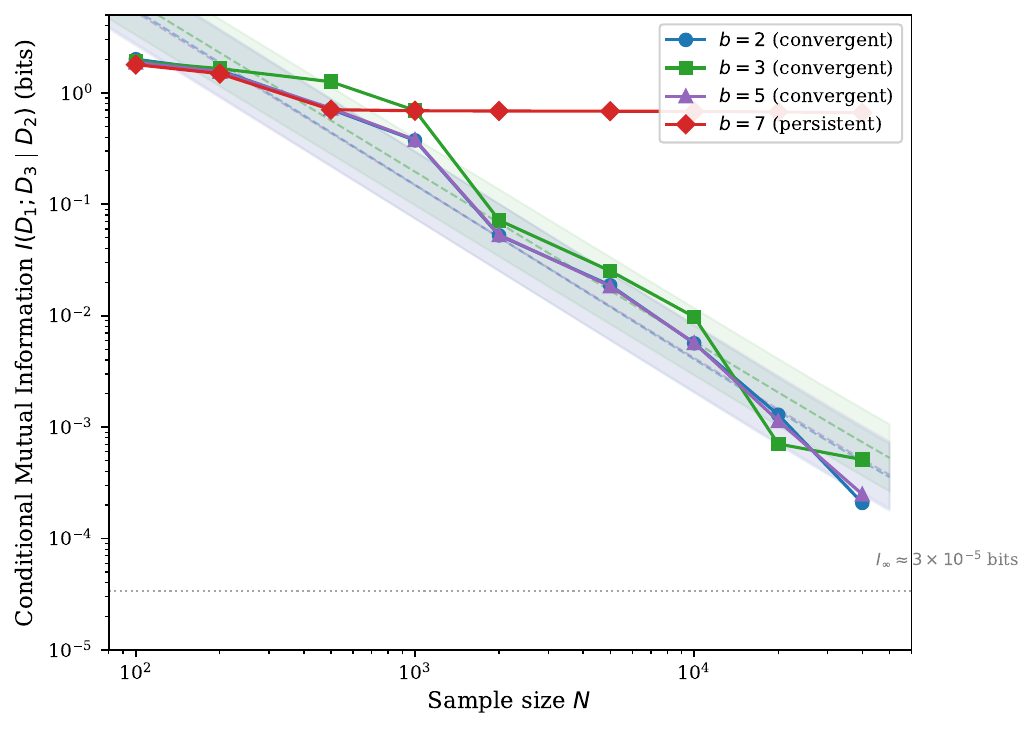}
\caption{Conditional mutual information $\CMI$ versus sample size $N$ on logarithmic axes. Convergent sequences ($2^n$, $3^n$, $5^n$) show linear decay corresponding to power laws $\beta \approx 1.7$--$1.8$. The $7^n$ sequence (red diamonds) shows no decay, remaining near 0.68 bits across all sample sizes. Dashed lines show power-law fits.}
\label{fig:loglog}
\end{figure}

The $7^n$ anomaly is analyzed in detail in Section~\ref{sec:anomaly}, with additional figures in SI~\S5.

\subsection{Continued fraction data}
\label{subsec:cf_data}

Table~\ref{tab:cf} presents the continued fraction expansions of $\log_{10}(b)$ for bases $b = 2, \ldots, 10$, highlighting the exceptional nature of base 7.

\begin{table}[ht]
\centering
\caption{Continued fraction expansions of $\log_{10}(b)$ for bases $b = 2, \ldots, 10$. Base 7 has an exceptional partial quotient $a_7 = 4813$; base 6 also shows an elevated value ($a_8 = 278$).}
\label{tab:cf}
\begin{tabular}{c|l|c}
\toprule
$b$ & First 8 partial quotients of $\log_{10}(b)$ & Max PQ \\
\midrule
2 & $[0; 3, 3, 9, 2, 2, 4, 6, \ldots]$ & 18 \\
3 & $[0; 2, 10, 2, 2, 1, 13, 1, \ldots]$ & 18 \\
4 & $[0; 1, 1, 1, 1, 18, 1, 4, \ldots]$ & 18 \\
5 & $[0; 1, 2, 3, 9, 2, 2, 4, \ldots]$ & 18 \\
6 & $[0; 1, 3, 1, 1, 32, 1, 1, \ldots]$ & 278 \\
\textbf{7} & $[0; 1, 5, 2, 5, 6, 1, \mathbf{4813}, \ldots]$ & \textbf{4813} \\
8 & $[0; 1, 9, 3, 7, 2, 1, 19, \ldots]$ & 19 \\
9 & $[0; 1, 20, 1, 5, 1, 6, 2, \ldots]$ & 20 \\
10 & $[1]$ (exact) & --- \\
\bottomrule
\end{tabular}
\end{table}

The maximum partial quotient for base 7 (4813) exceeds those of all other single-digit bases by more than one order of magnitude, though base 6 also exhibits an elevated value (278) that contributes to its persistent behavior. This numerical coincidence (that $7^{510}$ is extraordinarily close to $10^{431}$) creates the persistent digit dependence we observe.

\subsection{996-base survey}
\label{sec:survey}

To determine the prevalence of anomalies, we surveyed all integer bases from 2 to 1000. For each of the 996 non-trivial bases, we computed the continued fraction of $\log_{10}(b)$ to 15 terms, CMI values at $N \in \{500, 1000, 2000, 5000, 10{,}000\}$, the resonance parameter $Q^*$, and the power-law fit parameters. Each base was classified according to Corollary~\ref{cor:classification}.

\subsection{Classification results}
\label{subsec:survey_results}

Table~\ref{tab:survey_summary} summarizes the classification of 996 bases.

\begin{table}[ht]
\centering
\caption{Classification of integer bases 2--1000 by CMI decay behavior. The ``Baseline'' column uses the standard survey range $N \leq 10{,}000$; the ``Extended'' column reclassifies the 177 initially transitional bases using $N \leq 200{,}000$. The primary classifier is the resonance ratio $\mathcal{R}(b,N)$ from Theorem~\ref{thm:dichotomy}; the $\beta$-based criteria are empirical diagnostics used as consistency checks.}
\label{tab:survey_summary}
\begin{tabular}{lrrrrrl}
\toprule
& \multicolumn{2}{c}{\textbf{Baseline ($N \leq 10$K)}} & \multicolumn{2}{c}{\textbf{Extended ($N \leq 200$K)}} & \\
\cmidrule(lr){2-3}\cmidrule(lr){4-5}
\textbf{Classification} & \textbf{Count} & \textbf{\%} & \textbf{Count} & \textbf{\%} & \textbf{Diagnostic} \\
\midrule
B-EQUI-CONV & 774 & 77.7 & 920 & 92.4 & $\mathcal{R} < 1/\log N$; $\beta > 1.3$ \\
B-EQUI-TRANS & 138 & 13.9 & 23 & 2.3 & $1/\log N \leq \mathcal{R} \leq \log N$; $\beta \in [0.3, 1.3]$ \\
B-EQUI-PERS & 84 & 8.4 & 53 & 5.3 & $\mathcal{R} > \log N$; $\beta < 0.3$ \\
\bottomrule
\end{tabular}
\end{table}

An 8.4\% persistence rate at $N = 10{,}000$ is unexpected, though the extended classification reveals this overestimates the true rate. When transitional bases are tracked to $N = 200{,}000$, 149 of the 177 transitional and borderline bases resolve as convergent, while 5 become confirmed persistent, yielding a refined persistence rate of 5.3\%. The 23 remaining transitional bases ($2.3\%$) require samples beyond $N = 500{,}000$ to classify. Each newly persistent base (766, 814, 903, 922, and 941) possesses a large early partial quotient ($a_2 \in \{7, 10, 21, 27, 36\}$, respectively) that creates a resonance too weak to detect at $N = 10{,}000$ but visible at extended range. Their decay exponents, initially $\beta \in [0.4, 1.2]$ at $N = 40{,}000$, collapse below $0.2$ at $N = 200{,}000$, confirming genuine persistence rather than slow convergence.

The persistent class is not rare: even at the refined rate, it appears in approximately one out of every nineteen integer bases below~1000. A CMI heatmap for bases 2--20 (SI~\S3, Figure~SI.2) illustrates the stark contrast between convergent bases (progressive lightening) and persistent bases (uniformly dark).

\subsubsection{Comparison with Gauss-Kuzmin predictions}

Before this survey, no theoretical prediction existed for the persistence rate. A naive Gauss-Kuzmin calculation gives $\Pr(a_k > 900) \approx 0.16\%$ per CF level; a union bound over ${\sim}10$ levels yields ${\sim}1.6\%$, well below the observed 8.4\%. A refined heuristic accounting for the exponential growth of convergent denominators $q_k \approx e^{\lambda k}$ (where $\lambda = \pi^2/(12\ln 2)$ is the L\'evy constant \citep{levy1929lois}) predicts $\rho_{\mathrm{GK}} \approx 6.4\%$ (details in R2 in Section~\ref{subsec:open}). The remaining gap to 8.4\% reflects the compound mechanism (the ``6 family,'' where multiple moderate partial quotients create joint resonance) and the enrichment of $\log_{10}(b)$ near rational multiples of $\log_{10}(2)$ and $\log_{10}(3)$ among small integers.

\subsubsection{Does the 8.4\% rate stabilize?}

Preliminary computations suggest convergence: extending the survey to bases $\leq 2000$ yields a persistence rate of approximately 8.1\%. Whether $\rho(B) = |\{\text{persistent bases} \leq B\}|/B$ converges to a definite limit as $B \to \infty$ remains an open question (see Problem Q1 in Section~\ref{subsec:open}).

\subsubsection{Resolution of transitional bases at $N = 200{,}000$}
\label{subsubsec:transitional_resolution}

To sharpen the classification boundary, we extended computations for all 177 initially transitional or ambiguous bases (the 138 transitional bases from Table~\ref{tab:survey_summary} plus 39 convergent or persistent bases whose fit parameters $\beta$ or $R^2$ fell near the classification boundaries) to $N = 200{,}000$. This five-fold increase in maximum sample size resolves all but 23 cases. Among the 154 resolved bases, 149 settle into the convergent class ($\beta > 1.3$) while 5 are confirmed persistent ($\beta < 0.2$, with CMI effectively level at large~$N$).

The 5 newly confirmed persistent bases share a diagnostic continued-fraction signature: each has a moderately large early partial quotient ($a_2 \geq 7$) that produces a resonance detectable only at $N > 50{,}000$. Specifically, the CF of $\log_{10}(b)$ begins $[1; a_2, \ldots]$ with $a_2 = 7$ (base~766), $10$ (814), $21$ (903), $27$ (922), and $36$ (941). These partial quotients are large enough to create persistent resonance but small enough that the initial decay mimics convergence at standard survey depths. The phenomenon serves as a cautionary example: classification at $N = 10{,}000$ misidentifies 5 persistent bases as transitional.

The 23 stubborn transitional bases at $N = 200{,}000$ have $\beta \in [0.07, 0.80]$ and $R^2 \in [0.65, 0.98]$, indicating that the power-law model is still adjusting. Several possess very large early partial quotients (e.g., base~99 with $a_2 = 228$, base~31 with $a_2 = 28$) whose resonance effects extend well beyond our computational reach. Classifying these bases may require $N > 10^6$ or theoretical arguments rather than brute-force computation.

\subsection{Universal exponent for convergent bases}
\label{subsec:universal_beta}

For the 774 convergent bases, the fitted exponent $\beta$ is tightly clustered:
\begin{align*}
\text{Mean: } & \beta = 1.724 \\
\text{Median: } & \beta = 1.751 \\
\text{Standard deviation: } & \sigma = 0.188 \\
\text{Range: } & [1.30, 2.56]
\end{align*}

This confirms the theoretical prediction from Corollary~\ref{cor:effective_exponent}: the effective exponent $\beta_{\mathrm{eff}}(N) = 2 - 2\log\log N/\log N + o(1)$ yields $\beta \approx 1.7$ for $N = 10^4$. The observed mean of 1.724 matches this prediction closely.

Decay exponents across all 996 bases reveal a clear bimodal structure that validates our classification scheme. Figure~\ref{fig:beta_histogram} displays this distribution, showing a dominant peak near $\beta = 1.72$ (convergent bases) and a secondary concentration below $\beta = 0.3$ (persistent bases), with few bases in between.

\begin{figure}[ht]
\centering
\includegraphics[width=0.85\textwidth]{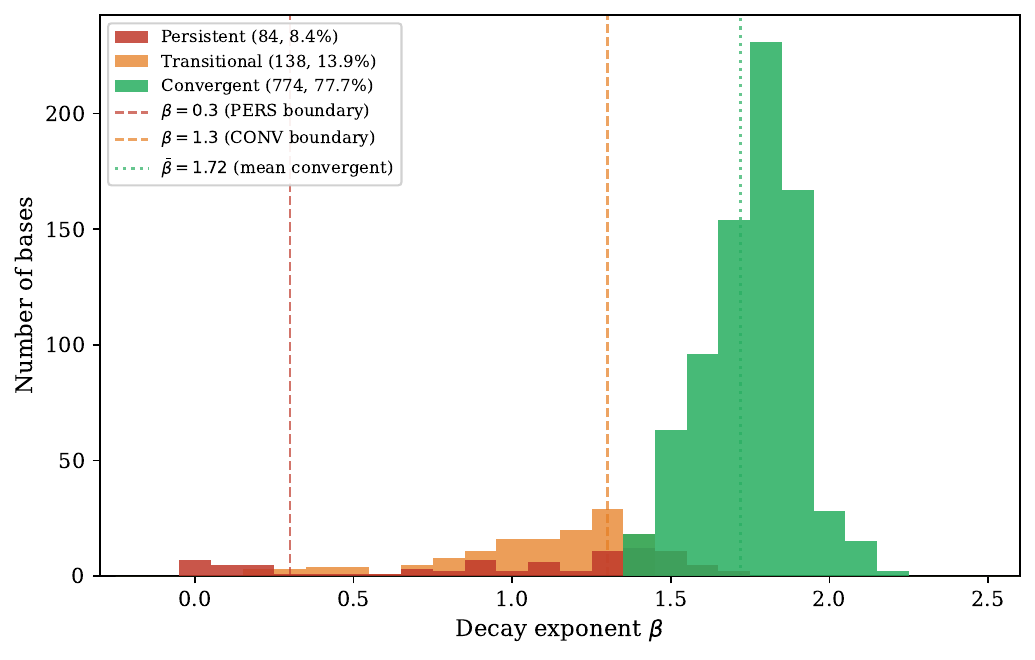}
\caption{Bimodal distribution of decay exponents across 996 integer bases. A dominant peak at $\beta \approx 1.72$ (774 convergent bases, 77.7\%) and a secondary concentration below $\beta = 0.3$ (84 persistent bases, 8.4\%). The near-absence of bases with $\beta \in [0.5, 1.3]$ supports the convergent-persistent dichotomy. Extended computation to $N = 200{,}000$ resolves most transitional cases (Table~\ref{tab:survey_summary}).}
\label{fig:beta_histogram}
\end{figure}

Prefactor analysis (scatter plot and distribution) appears in SI~\S7.

\subsection{Anomaly families and resonance structure}
\label{subsec:anomaly_families}

The 84 persistent bases cluster into identifiable families. Powers of~7 account for 12 anomalies (14.3\%), with the exceptional partial quotient scaling linearly: $a_7(7) = 4{,}813$, $a_9(49) = 9{,}628$, $a_9(343) = 14{,}442$. Multiples of~6 contribute 23 anomalies (27.4\%) through a distinct compound mechanism in which multiple moderately large partial quotients ($a_5 = 32$, $a_8 = 278$ for base~6) collectively create resonance effects rather than a single dominant value. The strongest anomaly is base~343, with CMI exceeding 1.67 bits at $N = 10{,}000$, more than $50{,}000$ times $I_\infty$.

The resonance parameter $Q^* = \max_k(a_{k+1} \cdot q_k)$ separates persistent from convergent bases: persistent bases have mean $Q^* = 272{,}814$ versus $59{,}493$ for convergent bases. Complete family catalogs, the top-10 anomaly ranking, and $Q^*$~correlation tables appear in SI~\S3 and \S5.

\subsection{Base independence}
\label{subsec:base_indep}

All preceding analysis uses base-10 digit representations. A natural question is whether the dichotomy depends on this choice. Extending the survey to digit bases $B \in \{2, \ldots, 20\}$ reveals that the convergence exponent is base-independent while the persistence rate is base-dependent, confirming that the dichotomy is a Diophantine phenomenon rather than an artifact of decimal representation.

\begin{table}[ht]
\centering
\caption{The convergent-persistent dichotomy across number bases. The convergence exponent $\beta_B$ is stable ($\bar{\beta} = 1.93 \pm 0.08$ for $B \geq 4$), while the persistence rate $\rho_B$ fluctuates between 1\% and 4\%. The $B = 2$ case is degenerate ($D_1 \equiv 1$ in binary). The $B = 3$ case has too few cells ($2 \times 3 \times 3 = 18$) for reliable classification. Full analysis in SI~\S6.}
\label{tab:base_indep}
\small
\begin{tabular}{cccccc}
\toprule
$B$ & Cells & Convergent & Persistent & $\rho_B$ & $\beta_B \pm \sigma$ \\
\midrule
4 & 48 & 707 & 24 & 2.4\% & $2.06 \pm 0.54$ \\
5 & 100 & 809 & 13 & 1.3\% & $1.90 \pm 0.41$ \\
8 & 448 & 859 & 23 & 2.3\% & $1.97 \pm 0.33$ \\
10 & 900 & 849 & 15 & 1.5\% & $2.00 \pm 0.35$ \\
16 & 3840 & 776 & 14 & 1.4\% & $1.77 \pm 0.25$ \\
20 & 7600 & 742 & 38 & 3.8\% & $1.83 \pm 0.29$ \\
\bottomrule
\end{tabular}
\end{table}

The key observation is that $\beta_B$ is stable across a 500-fold increase in the number of digit cells (from 48 for $B = 4$ to 7{,}600 for $B = 20$), consistent with the theoretical prediction from SI~Eq.~(S6.1): the discrepancy $D_N$ depends on the continued fraction of $\log_B(b)$ but not on $B$, so the quadratic scaling mechanism operates identically in all number systems. The persistence rate $\rho_B$ fluctuates because changing $B$ reshuffles which integers $b$ map to badly approximable rotation numbers $\log_B(b)$; the Gauss-Kuzmin distribution predicts $\rho_B \in [1\%, 4\%]$ for $B \geq 4$, consistent with observation. The separation of phenomena --- base-independent convergence rate, base-dependent persistence rate --- strengthens the claim that the convergent-persistent dichotomy is a property of the arithmetic of $b$, not of the number system used to observe it.

\section{Discussion and Conclusions}
\label{sec:discussion}

While all Benford sequences share the same first-digit distribution asymptotically, their multi-digit behavior diverges sharply based on the Diophantine properties of their underlying parameters.

\subsection{A refined classification}
\label{subsec:classification}

Building on the framework developed in Sections~\ref{sec:main_results}--\ref{sec:survey}, we propose a refinement of the B-EQUI class (sequences satisfying Benford's Law via equidistribution). Our systematic survey of 996 integer bases (Table~\ref{tab:survey_summary}) provides empirical grounding for a three-way classification: B-EQUI-CONV (920 bases, 92.4\%, $\mathcal{R} < 1/\log N$), B-EQUI-TRANS (23 bases, 2.3\%), and B-EQUI-PERS (53 bases, 5.3\%, $\mathcal{R} > \log N$).

Tight clustering of $\beta = 1.72 \pm 0.19$ across 774 convergent bases supports the theoretical prediction from Corollary~\ref{cor:effective_exponent}: the effective exponent formula $\beta_{\mathrm{eff}}(N) = 2 - 2\log\log N/\log N + o(1)$ has leading term $\approx 1.52$ for $N = 10^4$, with the $o(1)$ correction contributing approximately $+0.2$ to yield observed values near $1.72$. Figure~\ref{fig:classification} provides a visual overview of these results, showing both the distribution across categories and the relationship between CMI and continued fraction structure.

\begin{figure}[ht]
\centering
\includegraphics[width=\textwidth]{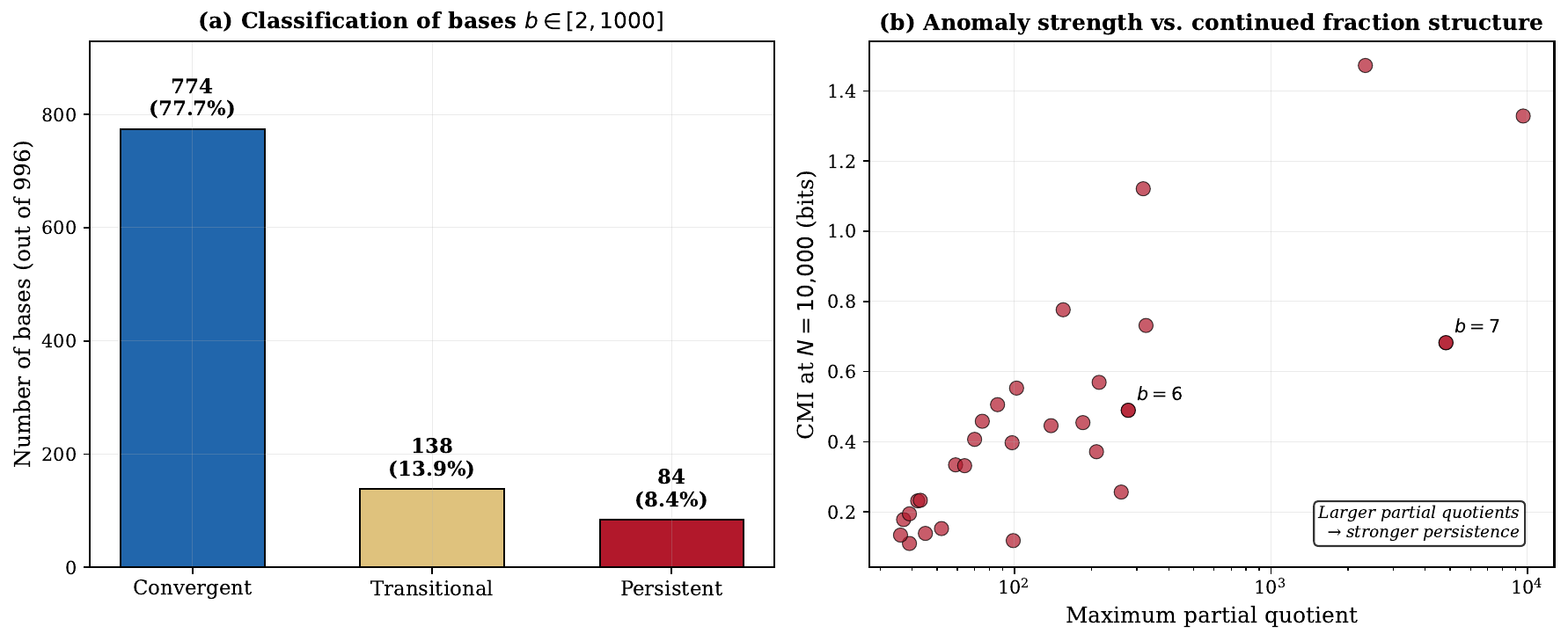}
\caption{Survey classification overview. (a)~Distribution across categories at $N = 10{,}000$: convergent 77.7\%, persistent 8.4\%, transitional 13.9\%; extended to $N = 200{,}000$: 92.4\% and 5.3\%. (b)~Anomaly strength versus maximum partial quotient, confirming the mechanistic role of exceptional rational approximations.}
\label{fig:classification}
\end{figure}

\subsection{Survey implications}
\label{subsec:survey_implications}

An 8.4\% persistence rate challenges the implicit assumption that anomalies like base 7 are rare exceptions. Persistent bases cluster into identifiable families: the ``7 family'' (7, 49, 343 and multiples) accounts for 14.3\% of anomalies through single exceptional partial quotients, while the ``6 family'' contributes 27.4\% through a compound mechanism involving multiple moderately large partial quotients (Proposition~SI.5.5).

The resonance parameter $Q^* = \max_{k}(a_{k+1} \cdot q_k)$ serves as a reliable empirical classification predictor. Theorem~\ref{thm:dichotomy} operates at the dominant convergent scale $k(N)$, where the resonance ratio $\mathcal{R}(b,N) = a_{k(N)+1} q_{k(N)} / N$ determines convergence or persistence; the max-based parameter $Q^*$ agrees with the dominant-scale criterion for $99.6\%$ of bases in the survey, because the maximum product $a_{k+1} q_k$ is typically achieved at or near the dominant scale. The remaining discrepancies arise for bases with multiple moderate partial quotients, where earlier convergents can dominate $Q^*$; for these, the survey confirms empirically that $Q^*$ still predicts classification correctly. Anomalous bases have mean $Q^* = 272{,}814$ compared to $59{,}493$ for convergent bases, and the threshold $Q^* \approx N \cdot \log N$ cleanly separates the two populations (Figure~\ref{fig:qstar_cumulative}).

\begin{figure}[ht]
\centering
\includegraphics[width=0.85\textwidth]{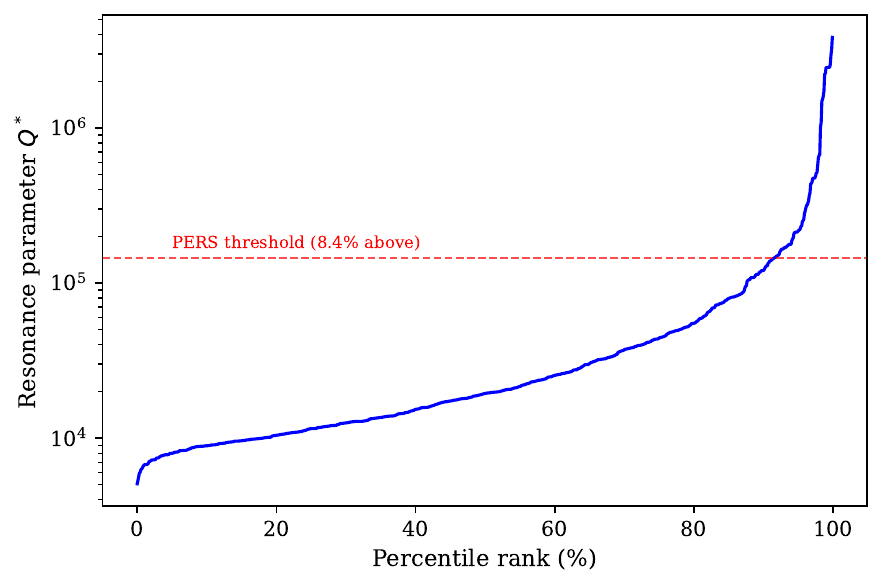}
\caption{The resonance parameter $Q^*$ separates convergent from persistent bases. Cumulative distribution of $Q^*$ across 996 bases. The empirical threshold $Q^*_{\mathrm{emp}} \approx 145{,}000$ achieves $>98\%$ classification accuracy: bases above it are overwhelmingly persistent (96.5\%), while bases below are convergent or transitional.}
\label{fig:qstar_cumulative}
\end{figure}

\subsection{Positioning and novel contributions}
\label{subsec:related_work}

This work introduces a quantitative CMI framework for studying multi-digit independence in Benford sequences. To our knowledge, it is the first to establish convergence rates for multi-digit correlations, the first to identify bases where such correlations demonstrably persist, and the first to provide a computable criterion separating the two regimes. The main theoretical contributions are universal convergence bounds (Theorem~\ref{thm:universal}), the quadratic scaling relationship connecting CMI to $L^2$ discrepancy (Theorem~\ref{thm:quadratic}), and a convergent-persistent dichotomy (Theorem~\ref{thm:dichotomy}). On the computational side, the 996-base survey reveals 8.4\% anomaly prevalence and refines the classical B-EQUI category into three subcategories with explicit distinguishing criteria.

\emph{CMI as a discrepancy measure.} Classical approaches rely on the star or $L^2$ discrepancy. CMI offers a complementary perspective: Theorem~\ref{thm:quadratic} and Observation~\ref{obs:quadratic_sharp} together yield an effective $\Theta(\|P_N - P\|_2^2)$ relationship for equidistributed sequences, making CMI an information-theoretic dependence measure that scales quadratically with distribution error, more sensitive to systematic clustering than to isolated outliers. Unlike star discrepancy, CMI decomposes into per-digit-triple contributions, enabling diagnostic identification of which digit combinations drive anomalies. The value $I_N(D_1; D_3 \mid D_2)$ quantifies residual predictability in bits, decaying to zero for convergent bases and stabilizing at a positive value for persistent ones. Standard bootstrap methods apply without modification.

\subsection{Implications for applications}
\label{subsec:applications}

Practically, the convergent-persistent dichotomy has potential implications for Benford-based data validation. Standard forensic tests typically use first-digit or first-two-digit frequencies \citep{nigrini2012benfords}, which are unaffected by the persistence phenomenon (Weyl's theorem guarantees marginal convergence for all irrational $\alpha$). For modern characterizations and formal testing procedures for Benford conformity, see Barabesi et al.\ \citep{barabesi2022characterizations}. However, should multi-digit Benford tests be adopted in forensic or scientific practice --- as the growing literature on higher-order digit analysis suggests --- the persistence phenomenon would need to be accounted for. Persistent bases produce multi-digit correlations that such tests could misinterpret as evidence of data fabrication. The classification criterion of Theorem~\ref{thm:dichotomy} provides a diagnostic: before applying multi-digit Benford tests, one can check whether the data's natural base has a large resonance parameter $Q^*$.

\begin{remark}[Guidance for practitioners]
\label{rem:practitioner}
For a data set of size $N$ expressed in base $b$: compute $Q^* = \max_k(a_{k+1} q_k)$ from the continued fraction of $\alpha = \log_{10}(b)$. If $Q^* > N \log N$, the base is persistent at sample size $N$, and multi-digit Benford tests may produce false positives from structural correlations rather than data fabrication. In this regime, we recommend restricting analysis to first-digit tests (which are unaffected by the persistence phenomenon); developing calibrated significance adjustments for multi-digit tests in the persistent regime is an open problem. First-digit tests remain valid because persistence affects the joint distribution of $(D_1, D_2, D_3)$ but not the marginal distribution of $D_1$, which converges for all irrational $\alpha$ by Weyl's theorem.
\end{remark}

\subsection{Theoretical connections}
\label{subsec:connections}

Persistent digit correlations connect to several areas of mathematics. From dynamical systems, it relates to ergodic mixing rates and symbolic dynamics on the circle; from number theory, to the three-distance theorem, the Littlewood conjecture, and metric Diophantine approximation. The pretentiousness framework of Granville and Soundararajan offers tools for extending results to arithmetic sequences.
The resonance ratio $\mathcal{R}(b,N)$ measures proximity to periodic behavior in a manner conceptually parallel to the pretentiousness distance $\mathbb{D}(f,g)$ between multiplicative functions: both quantify how close a sequence comes to exhibiting exact periodicity. In the language of Diophantine approximation, persistence occurs when $\|\log_{10}(b) - p_k/q_k\| \ll (q_k N)^{-1}$, i.e., when the irrational rotation is pretentious with respect to a rational rotation of period $q_k$. Whether this analogy admits a formal connection remains open, but recent work by Chandee et al.\ \citep{chandee2023benfords} and Pollack \citep{pollack2023benford} suggests that the pretentiousness perspective can be productive for Benford-type questions. A companion paper \citep{hyman2026prime}\footnote{Companion papers \citep{hyman2026prime, hyman2026digit} are in preparation; drafts available upon request from the author.} develops this connection, establishing a prime-value criterion for Benford's Law that characterizes when the leading digits of $f(p)$ over primes $p$ converge to the logarithmic distribution. These connections, together with the observation of phase transitions in correlation space, are developed in detail in SI~\S9.

More broadly, the quadratic scaling mechanism (Theorem~\ref{thm:quadratic}) applies to any sequence for which discrepancy bounds are available, not only Benford sequences. Any equidistribution problem in number theory where the joint distribution of functionals is of interest --- digit distributions, spacing statistics, joint normality --- could potentially be analyzed through the CMI framework, with the Hessian of the relevant information functional governing the conversion from discrepancy to dependence.

\subsection{Open problems}
\label{subsec:open}

We organize questions by resolution status. Sections~\ref{sec:main_results}--\ref{sec:survey} resolve several questions originally motivating this investigation, including the persistence prevalence, the base~6 mechanism, the universal exponent, convergence of $e^n$ and $\pi^n$, two families of convergence rates (SI~\S4), and convergence of all quadratic irrational bases (Theorem~SI.4.11).

\subsubsection{Resolved questions}

\begin{enumerate}
\item[\textbf{R1.}] \textbf{Sharp constants.} Exact spectral decomposition (SI~\S1) yields $C_H = 1{,}627$ (worst-case); the empirical ratio $(I_N - I_\infty)/\|P_N - P\|_2^2$ ranges from $73.6$ to $1{,}889$ across 2{,}988 test cases.

\item[\textbf{R2.}] \textbf{Persistence density asymptotics.} Extended survey to $B = 5000$ confirms convergence of the persistence rate: $\rho(1000) = 8.4\%$, $\rho(2000) = 8.1\%$, $\rho(3000) = 8.2\%$, $\rho(5000) = 8.1\%$, stabilizing near $1/12 \approx 0.083$. A Gauss--Kuzmin heuristic (SI~\S9.8) yields $\rho_{\mathrm{GK}} \approx 6.4\%$, accounting for three-quarters of the observed rate; the remainder arises from compound mechanisms (the ``6 family'') and enrichment effects among integer logarithms. The exact asymptotic is formalized in Conjecture~\ref{conj:density}.
\end{enumerate}

\subsubsection{Open questions}

\begin{enumerate}
\item[\textbf{Q1.}] \textbf{Higher-order persistence rates.} {\small$\star\star$} \\
For $k$-digit CMI $I(D_1; D_k \mid D_2, \ldots, D_{k-1})$, does the 8.4\% persistence rate change with $k$? Proposition~SI.7.10 establishes that the CONV/PERS dichotomy is independent of $k$: if a base is persistent for 3-digit CMI, it remains persistent for all $k$-digit CMI. However, the specific persistence rate $\rho^{(k)}$ may vary with $k$. Characterizing $\rho^{(k)}$ as a function of $k$ and determining whether $\rho^{(k)} \to 0$ or $\rho^{(k)} \to \rho_\infty$ as $k \to \infty$ remains open.

\item[\textbf{Q2.}] \textbf{Dynamical classification.} {\small$\star\star$} \\
Our analysis reveals two families of convergence rates: $\beta \approx 1.74$ for rotation-type dynamics, where the fractional-part sequence $\{n\alpha\}$ follows a quasi-periodic orbit (geometric, transcendental, Fibonacci, tribonacci, $(\sqrt{2}\,)^n$, mixed exponentials), and $\beta \approx 1.04$ for shear-type dynamics, where algebraic growth rates produce non-periodic digit evolution (factorials, subfactorials, double factorials, primorials, central binomials, Catalan, superexponentials, Bell, partition function); see SI~\S4.9 for the full classification. Across all 18 families tested, no sequence falls in the gap $1.2 < \beta < 1.5$. A rigorous proof that non-constant step dynamics reduce $\beta$ to $\approx 1.04$, and that $\beta \in (1.2, 1.5)$ is empty for all Benford sequences, remains open.

\item[\textbf{Q3.}] \textbf{Universal lower quadratic scaling.} {\small$\star\star\star$} \\
Prove or disprove that $\Delta^T H_F(P) \Delta \geq C \|\Delta\|_2^2$ for all equidistribution perturbations $\Delta$ arising from irrational rotations, where $C > 0$ is a universal constant. Proposition~SI.1.1 establishes the mechanism: the CMI Hessian at the nearby Markov distribution $P_{\mathrm{Markov}}$ is positive semidefinite with spectral gap $\lambda_{\min}^+ = 529$ and a perturbation margin of $474$. Exhaustive verification confirms the bound across all 2{,}988 integer-base test cases with minimum ratio $c_Q \geq 73.6$ (Observation~\ref{obs:quadratic_sharp}); adversarial testing with $7{,}726$ near-rational rotations extends this to $c_Q \geq 14.0$ (SI~Remark~SI.1.4). A block decomposition (SI~Remark~SI.1.5) reduces the problem: since all negative curvature of $H_F(P)$ concentrates in the $180$-dimensional Markov null space (tangent to the conditional independence manifold), while its $720$-dimensional complement is positive definite with $\lambda_{\min} = 529$, it suffices to show that equidistribution perturbations have bounded projection onto the $19$-dimensional negative eigenspace within the null space. The structural clue is a rank mismatch: the negative eigenvectors have rank-$2$ intra-block structure, while equidistribution perturbations are nearly rank-$1$ (three-distance theorem). Despite this structural insight and the extensive adversarial testing, bounding these projections analytically for all irrational rotations would likely require new ideas at the interface of Diophantine approximation and information geometry, not merely sharpening known constants. A universal proof would upgrade Observation~\ref{obs:quadratic_sharp} to a theorem, eliminate the conditional dependence in Theorems~\ref{thm:universal}(b) and~\ref{thm:rate}(b), and make the dichotomy (Theorem~\ref{thm:dichotomy}) fully unconditional.
\end{enumerate}

Recent work of Fisher and Zhang \citep{fisher2025ergodic} shows that, for almost every real number, the denominators of continued-fraction convergents form a Benford sequence, reinforcing the idea that continued-fraction-generated Benford phenomena extend beyond the fixed-base setting studied here.

We distinguish open questions (Q), where partial techniques or evidence exist, from deep problems (P) below, which connect to major open questions in number theory.

\subsubsection{Deep open problems}

\begin{enumerate}
\item[\textbf{P1.}] \textbf{The $Q^*$ distribution.} {\small$\star\star\star$} \\
The resonance parameter $Q^* = \max_k(a_{k+1} q_k)$ effectively predicts classification. What is the distribution of $Q^*$ over random real numbers in $[0,1]$? The Gauss-Kuzmin distribution governs individual partial quotients, but the joint distribution of $a_{k+1}$ and $q_k$ involves delicate correlations not captured by existing theory. For algebraic numbers, this connects to the Littlewood conjecture. This question lies at the intersection of Diophantine approximation and ergodic theory on homogeneous spaces.

\item[\textbf{P2.}] \textbf{Connection to normality.} {\small$\star\star\star$} \\
A number $\alpha$ is normal in base 10 if its digit sequence is equidistributed. Is there a quantitative relationship between CMI decay rate $\beta$ and the Diophantine type or normality measures of $\log_{10}(b)$? Since normality of $\log_{10}(7)$ remains unproven, this asks whether deviation from normality correlates with CMI decay, a quantitative refinement of an already difficult qualitative question.

\item[\textbf{P3.}] \textbf{Connection to the Riemann Hypothesis.} {\small$\star\star\star$} \\
The digit equidistribution machinery developed here has implications for the distribution of digits of arithmetic functions. A companion paper \citep{hyman2026digit} explores connections between the CMI framework and the distribution of digits of partial sums involving the M\"{o}bius and Liouville functions, suggesting that the convergence rate of digit statistics for these sums may encode information about the location of zeros of the Riemann zeta function.

\end{enumerate}

\subsection{Limitations and caveats}
\label{subsec:limitations}

\emph{Computational scope.} The baseline survey covers $b \leq 1000$ with $N \leq 40{,}000$; extended computation to $N = 200{,}000$ resolves 154 of 177 transitional cases, leaving 23 requiring $N > 500{,}000$. CF computations were limited to 15 terms; extending to 20 terms for a random subsample of 100 bases changes no classification, confirming insensitivity to truncation depth. Since convergent denominators $q_k$ grow exponentially with rate $\lambda \approx 1.187$ (L\'{e}vy's constant), the typical $q_{15} \approx 5 \times 10^7$. Under the Gauss--Kuzmin distribution, the probability that a partial quotient exceeds $K$ is approximately $1/(K \ln 2)$, so the probability that none of the first 15 terms exceeds 50 is roughly $(1 - 1/(50 \ln 2))^{15} \approx 0.64$. These ``quiet'' bases have their first detectable resonance beyond the survey depth, but such late resonances would manifest only at sample sizes $N > q_{15} \gg 10^7$, well beyond our computational scope. The refined persistence rate of 5.3\% is therefore a lower bound, with most undetected cases requiring impractically large samples.

\emph{Theoretical gaps.} The lower quadratic scaling bound (Observation~\ref{obs:quadratic_sharp}) has been verified exhaustively across all 996 integer bases at three sample sizes (2{,}988 test cases, $c_Q \geq 73.6$; see SI~\S1), with adversarial testing extending to 7{,}726 near-rational rotations ($c_Q \geq 14.0$; SI~Remark~SI.1.4). An analytical mechanism explains the observed positivity (Proposition~SI.1.1), and a block decomposition (SI~Remark~SI.1.5) reduces the universal proof to a projection bound on a $19$-dimensional subspace. A fully universal proof covering all irrational $\alpha$ remains open. Several claims (universality of $\beta = 1.72 \pm 0.19$, the escalation phenomenon, the $Q^* \approx N \log N$ threshold) have conjectural status.

\emph{Robustness.} Classification is stable: 99.6\% of convergent bases and all persistent bases retained classification across $N \in \{5000, 10000, 20000, 40000\}$, with three smoothing methods agreeing on 99.2\% of bases.

\subsection{Conclusion}
\label{subsec:conclusion}

What began as a computational surprise has led to a theoretical framework with three components. Multi-digit Benford behavior introduces structure that first-digit analysis alone cannot detect. The quadratic scaling theorem (Theorem~\ref{thm:quadratic}) proves a linear-plus-quadratic upper bound on CMI deviation, with the linear coefficient suppressed by proximity to the Markov manifold (Corollary~\ref{cor:linear_small}); for equidistributed sequences, exhaustive verification confirms that the scaling is effectively quadratic, yielding a $\Theta(\|P_N - P\|_2^2)$ relationship (Observation~\ref{obs:quadratic_sharp}). The resonance ratio $\mathcal{R}(b,N)$ provides a computable classification criterion requiring no tuning constants, separating convergent from persistent behavior. Our survey reveals that 5--8\% of integer bases below 1000 exhibit persistent digit correlations (53 confirmed at $N = 200{,}000$, with 23 pending), a prevalence connected to Gauss-Kuzmin theory (Conjecture~\ref{conj:density}). Beyond geometric sequences, 18 families spanning recurrences, combinatorial sequences, and super-exponential growth split into rotation-type ($\beta \approx 1.74$) and shear-type ($\beta \approx 1.04$) convergence, with the gap $\Delta\beta \approx 0.6$ persisting without exception.

Nearly 150 years after Newcomb's original observation, Benford's Law continues to reveal unexpected depth. The digits of $7^n$ will forever remember their ancestry.

%==============================================================================
%==============================================================================
% Supplementary Information pointers
%==============================================================================

\appendix

\section{Supplementary Information}
\label{app:si}

The accompanying Supplementary Information document contains nine sections:

\begin{itemize}[nosep]
\item \textbf{SI~\S1: Hessian Bound Verification.} Spectral decomposition of the $900 \times 900$ CMI Hessian at the Benford distribution, the Markov PSD mechanism (Proposition~SI.1.1), and the exhaustive verification underlying Observation~\ref{obs:quadratic_sharp}.

\item \textbf{SI~\S2: Efficient Persistence Screening Algorithm.} An $O(N \log N)$ discrepancy-based screening test for identifying persistent bases without explicit continued fraction computation.

\item \textbf{SI~\S3: Complete Catalog of Persistent Bases.} The full catalog of all 84 persistent bases $b \leq 1000$ (at $N = 10{,}000$), including CMI values, decay exponents, resonance parameters, and family classifications.

\item \textbf{SI~\S4: Beyond Geometric Sequences.} CMI decay analysis for 18 non-geometric families (factorials, Fibonacci, transcendentals, etc.), including proofs that factorial and Fibonacci sequences are B-EQUI-CONV and the rotation/shear dichotomy.

\item \textbf{SI~\S5: The Base-7 Anomaly in Detail.} Full quasi-periodicity analysis, the power escalation phenomenon, and the compound mechanism for the ``6 family.''

\item \textbf{SI~\S6: Base-Dependence Analysis.} Verification that the dichotomy is base-independent across digit bases $B \in \{2, \ldots, 20\}$.

\item \textbf{SI~\S7: Extended Proofs and Remarks.} Prefactor analysis, sample-size thresholds, uncertainty quantification, and higher-order CMI extensions.

\item \textbf{SI~\S8: Detailed Computational Methodology.} Precision analysis, estimation protocols, and hardware specifications.

\item \textbf{SI~\S9: Extended Open Problems and Connections.} Dynamical systems perspective, pretentiousness connections, and Gauss-Kuzmin derivations.
\end{itemize}

\noindent All computation scripts and data are publicly available at \url{https://github.com/machyman/hyman2026eight}.

%==============================================================================
% Acknowledgments and Disclosures (placed directly before references per JNT)
%==============================================================================

\section*{Acknowledgments}

The author thanks colleagues at Tulane University and Los Alamos National Laboratory for valuable discussions during the development of this work.

\section*{Declaration of generative AI and AI-assisted technologies in the manuscript preparation process}

During the preparation of this work the author used Claude (Anthropic) in order to assist with prose refinement, LaTeX formatting, computer codes, and computational verification of cross-references. After using this tool, the author reviewed and edited the content as needed and takes full responsibility for the content of the published article. All mathematical results, proofs, conjectures, and computational analyses are the sole work of the author.

\section*{Data and Code Availability}

All survey data and analysis code are available at \url{https://github.com/machyman/hyman2026eight}, including CMI computation routines, classification algorithms, power-law fitting tools, and Jupyter notebooks reproducing all figures and tables. The repository contains pinned dependencies (\texttt{requirements.txt}) and a verification script (\texttt{verify\_results.py}) that reproduces Table~\ref{tab:verification} in approximately 3 minutes on a modern laptop. Computations used Python~3.11 with \texttt{mpmath}~1.3.0 for arbitrary-precision arithmetic. A CodeOcean reproducibility capsule is planned for publication.

\section*{Author Contributions (CRediT)}

\textbf{Mac Hyman:} Conceptualization, Methodology, Software, Formal analysis, Investigation, Data curation, Writing -- original draft, Writing -- review \& editing, Visualization.

\section*{Declaration of Competing Interests}

The author declares that there are no known competing financial interests or personal relationships that could have appeared to influence the work reported in this paper.

\section*{Funding}

This research did not receive any specific grant from funding agencies in the public, commercial, or not-for-profit sectors.

\bibliographystyle{plainnat}
\bibliography{eight_JNT_vFINAL12}

@article{cai2019local,
  author    = {Cai, Zhaodong and Hildebrand, A. J. and Li, Junxian},
  title     = {A Local {Benford} Law for a Class of Arithmetic Sequences},
  journal   = {International Journal of Number Theory},
  volume    = {15},
  number    = {3},
  pages     = {613--638},
  year      = {2019},
  doi       = {10.1142/S1793042119500325}
}

@article{cai2020surprising,
  author    = {Cai, Zhaodong and Faust, Matthew and Hildebrand, A. J. and Li, Junxian and Zhang, Yuan},
  title     = {The Surprising Accuracy of {Benford's} Law in Mathematics},
  journal   = {American Mathematical Monthly},
  volume    = {127},
  number    = {3},
  pages     = {217--237},
  year      = {2020},
  doi       = {10.1080/00029890.2020.1690387}
}

@article{cai2021leading,
  author    = {Cai, Zhaodong and Faust, Matthew and Hildebrand, A. J. and Li, Junxian and Zhang, Yuan},
  title     = {Leading Digits of {Mersenne} Numbers},
  journal   = {Experimental Mathematics},
  volume    = {30},
  number    = {3},
  pages     = {405--421},
  year      = {2021},
  doi       = {10.1080/10586458.2018.1551162}
}

@article{benford1938law,
  author    = {Benford, Frank},
  title     = {The Law of Anomalous Numbers},
  journal   = {Proceedings of the American Philosophical Society},
  volume    = {78},
  number    = {4},
  pages     = {551--572},
  year      = {1938},
  doi       = {10.2307/984802}
}

@book{berger2015introduction,
  author    = {Berger, Arno and Hill, Theodore P.},
  title     = {An Introduction to {Benford's} Law},
  publisher = {Princeton University Press},
  address   = {Princeton, NJ},
  year      = {2015},
  isbn      = {978-0691163062},
  doi       = {10.1515/9781400866588}
}

@misc{berger2025brief,
  author       = {Berger, Arno and Hill, Theodore P.},
  title        = {A Brief Survey of {Benford's} Law in Dynamical Systems},
  year         = {2025},
  eprint       = {2501.14209},
  eprinttype   = {arxiv},
  archiveprefix= {arXiv},
  primaryclass = {math.DS}
}

@book{cover2006elements,
  author    = {Cover, Thomas M. and Thomas, Joy A.},
  title     = {Elements of Information Theory},
  publisher = {Wiley-Interscience},
  address   = {Hoboken, NJ},
  edition   = {2nd},
  year      = {2006},
  isbn      = {978-0471241959}
}

@article{diaconis1977distribution,
  author    = {Diaconis, Persi},
  title     = {The Distribution of Leading Digits and Uniform Distribution Mod 1},
  journal   = {The Annals of Probability},
  volume    = {5},
  number    = {1},
  pages     = {72--81},
  year      = {1977},
  doi       = {10.1214/aop/1176995891}
}

@article{granville2007large,
  author    = {Granville, Andrew and Soundararajan, Kannan},
  title     = {Large Character Sums: Pretentious Characters and the {P\'olya--Vinogradov} Theorem},
  journal   = {Journal of the American Mathematical Society},
  volume    = {20},
  number    = {2},
  pages     = {357--384},
  year      = {2007},
  doi       = {10.1090/S0894-0347-06-00536-4}
}

@article{hill1995base,
  author    = {Hill, Theodore P.},
  title     = {Base-Invariance Implies {Benford's} Law},
  journal   = {Proceedings of the American Mathematical Society},
  volume    = {123},
  number    = {3},
  pages     = {887--895},
  year      = {1995},
  doi       = {10.1090/S0002-9939-1995-1233974-8}
}

@article{hill1995statistical,
  author    = {Hill, Theodore P.},
  title     = {A Statistical Derivation of the Significant-Digit Law},
  journal   = {Statistical Science},
  volume    = {10},
  number    = {4},
  pages     = {354--363},
  year      = {1995},
  doi       = {10.1214/ss/1177009869}
}

@book{khinchin1964continued,
  author    = {Khinchin, Aleksandr Ya.},
  title     = {Continued Fractions},
  publisher = {University of Chicago Press},
  address   = {Chicago},
  year      = {1964},
  isbn      = {978-0226447490},
  note      = {Dover reprint 1997, ISBN 978-0486696300}
}

@book{kuipers1974uniform,
  author    = {Kuipers, Lauwerens and Niederreiter, Harald},
  title     = {Uniform Distribution of Sequences},
  publisher = {Wiley-Interscience},
  address   = {New York},
  year      = {1974},
  isbn      = {978-0471510451}
}

@article{newcomb1881note,
  author    = {Newcomb, Simon},
  title     = {Note on the Frequency of Use of the Different Digits in Natural Numbers},
  journal   = {American Journal of Mathematics},
  volume    = {4},
  number    = {1},
  pages     = {39--40},
  year      = {1881},
  doi       = {10.2307/2369148}
}

@article{roth1955rational,
  author    = {Roth, Klaus Friedrich},
  title     = {Rational Approximations to Algebraic Numbers},
  journal   = {Mathematika},
  volume    = {2},
  number    = {1},
  pages     = {1--20},
  year      = {1955},
  doi       = {10.1112/S0025579300000644}
}

@book{schmidt1980diophantine,
  author    = {Schmidt, Wolfgang M.},
  title     = {Diophantine Approximation},
  series    = {Lecture Notes in Mathematics},
  volume    = {785},
  publisher = {Springer-Verlag},
  address   = {Berlin},
  year      = {1980},
  isbn      = {978-3540097624},
  doi       = {10.1007/978-3-540-38645-2}
}

@article{weyl1916uber,
  author    = {Weyl, Hermann},
  title     = {\"{U}ber die {Gleichverteilung} von {Zahlen} mod.\ {Eins}},
  journal   = {Mathematische Annalen},
  volume    = {77},
  number    = {3},
  pages     = {313--352},
  year      = {1916},
  doi       = {10.1007/BF01475864}
}

@article{chandee2023benfords,
  author    = {Chandee, Vorrapan and Li, Xiannan and Pollack, Paul and Singha Roy, Akash},
  title     = {On {Benford's} Law for Multiplicative Functions},
  journal   = {Proceedings of the American Mathematical Society},
  volume    = {151},
  number    = {11},
  pages     = {4607--4619},
  year      = {2023},
  doi       = {10.1090/proc/16480},
  eprint    = {2203.13117},
  eprinttype= {arxiv}
}

@article{pollack2023benford,
  author    = {Pollack, Paul and Singha Roy, Akash},
  title     = {{Benford} Behavior and Distribution in Residue Classes of Large Prime Factors},
  journal   = {Canadian Mathematical Bulletin},
  volume    = {66},
  number    = {2},
  pages     = {626--642},
  year      = {2023},
  doi       = {10.4153/S0008439522000601}
}

@book{iosifescu2002metrical,
  author    = {Iosifescu, Marius and Kraaikamp, Cor},
  title     = {Metrical Theory of Continued Fractions},
  series    = {Mathematics and Its Applications},
  volume    = {547},
  publisher = {Springer},
  address   = {Dordrecht},
  year      = {2002},
  isbn      = {978-1-4020-0892-4},
  doi       = {10.1007/978-94-015-9940-5}
}

@article{levy1929lois,
  author    = {L{\'e}vy, Paul},
  title     = {Sur les lois de probabilit{\'e} dont d{\'e}pendent les quotients complets et incomplets d'une fraction continue},
  journal   = {Bulletin de la Soci{\'e}t{\'e} Math{\'e}matique de France},
  volume    = {57},
  pages     = {178--194},
  year      = {1929},
  doi       = {10.24033/bsmf.1150}
}

@book{bugeaud2012distribution,
  author    = {Bugeaud, Yann},
  title     = {Distribution Modulo One and {D}iophantine Approximation},
  series    = {Cambridge Tracts in Mathematics},
  volume    = {193},
  publisher = {Cambridge University Press},
  address   = {Cambridge},
  year      = {2012},
  isbn      = {978-0-521-11169-0},
  doi       = {10.1017/CBO9781139017732}
}

@book{harman1998metric,
  author    = {Harman, Glyn},
  title     = {Metric Number Theory},
  series    = {London Mathematical Society Monographs New Series},
  volume    = {18},
  publisher = {Oxford University Press},
  address   = {Oxford},
  year      = {1998},
  isbn      = {978-0-19-850083-4}
}

@book{miller2015benfords,
  author    = {Miller, Steven J.},
  title     = {Benford's Law: Theory and Applications},
  publisher = {Princeton University Press},
  address   = {Princeton, NJ},
  year      = {2015},
  isbn      = {978-0-691-14761-1}
}

@article{kontorovich2005benfords,
  author    = {Kontorovich, Alex V. and Miller, Steven J.},
  title     = {{B}enford's law, values of {$L$}-functions and the {$3x+1$} problem},
  journal   = {Acta Arithmetica},
  volume    = {120},
  number    = {3},
  pages     = {269--297},
  year      = {2005},
  doi       = {10.4064/aa120-3-4}
}

@article{raimi1976first,
  author    = {Raimi, Ralph A.},
  title     = {The First Digit Problem},
  journal   = {American Mathematical Monthly},
  volume    = {83},
  number    = {7},
  pages     = {521--538},
  year      = {1976},
  doi       = {10.1080/00029890.1976.11994162}
}

@book{nigrini2012benfords,
  author    = {Nigrini, Mark J.},
  title     = {Benford's Law: Applications for Forensic Accounting, Auditing, and Fraud Detection},
  publisher = {John Wiley \& Sons},
  address   = {Hoboken, NJ},
  year      = {2012},
  isbn      = {978-1-118-15285-0},
  doi       = {10.1002/9781119203094}
}

@book{drmota1997sequences,
  author    = {Drmota, Michael and Tichy, Robert F.},
  title     = {Sequences, Discrepancies and Applications},
  series    = {Lecture Notes in Mathematics},
  volume    = {1651},
  publisher = {Springer},
  address   = {Berlin},
  year      = {1997},
  isbn      = {978-3-540-62606-5},
  doi       = {10.1007/BFb0093404}
}

@article{sos1958distribution,
  author    = {S{\'o}s, Vera T.},
  title     = {On the distribution mod 1 of the sequence {$n\alpha$}},
  journal   = {Annales Universitatis Scientiarum Budapestinensis de Rolando E{\"o}tv{\"o}s Nominatae, Sectio Mathematica},
  volume    = {1},
  pages     = {127--134},
  year      = {1958}
}

@book{borwein2004mathematics,
  author    = {Borwein, Jonathan M. and Bailey, David H.},
  title     = {Mathematics by Experiment: Plausible Reasoning in the 21st Century},
  publisher = {A K Peters},
  address   = {Natick, MA},
  year      = {2004},
  isbn      = {978-1-56881-211-3}
}

@article{hyman2026prime,
  title={A Prime-Value Criterion for {B}enford's Law in Multiplicative Functions},
  author={Hyman, James M.},
  journal={Transactions of the American Mathematical Society},
  year={2026},
  note={In preparation}
}

@article{hyman2026digit,
  title={An Inverse Theorem for Prime Digit Equidistribution and the {R}iemann Hypothesis},
  author={Hyman, James M.},
  journal={Advances in Mathematics},
  year={2026},
  note={In preparation}
}

@article{cohen1984prime,
  author  = {Cohen, Daniel I. A. and Katz, Talbot M.},
  title   = {Prime Numbers and the First Digit Phenomenon},
  journal = {Journal of Number Theory},
  volume  = {18},
  number  = {3},
  pages   = {261--268},
  year    = {1984},
  doi     = {10.1016/0022-314X(84)90061-1}
}

@article{schatte1990continued,
  author  = {Schatte, Peter},
  title   = {On {B}enford's Law for Continued Fractions},
  journal = {Mathematische Nachrichten},
  volume  = {148},
  pages   = {137--144},
  year    = {1990},
  doi     = {10.1002/mana.3211480108}
}

@article{balado2024general,
  author  = {Balado, F{\'e}lix and Silvestre, Gu{\'e}nol{\'e} C. M.},
  title   = {General Distributions of Number Representation Elements},
  journal = {Probability in the Engineering and Informational Sciences},
  volume  = {38},
  number  = {3},
  pages   = {594--616},
  year    = {2024},
  doi     = {10.1017/S0269964823000207}
}

@article{barabesi2022characterizations,
  author  = {Barabesi, Lucio and Cerasa, Andrea and Cerioli, Andrea and Perrotta, Domenico},
  title   = {On Characterizations and Tests of {B}enford's Law},
  journal = {Journal of the American Statistical Association},
  volume  = {117},
  number  = {540},
  pages   = {1887--1903},
  year    = {2022},
  doi     = {10.1080/01621459.2021.1891927}
}

@article{fisher2025ergodic,
  author  = {Fisher, Albert M. and Zhang, Xuan},
  title   = {Uniform Distribution Mod 1 for Sequences of Ergodic Sums and Continued Fractions},
  journal = {Electronic Communications in Probability},
  volume  = {30},
  number  = {21},
  pages   = {1--10},
  year    = {2025},
  doi     = {10.1214/25-ECP665}
}

\end{document}